\documentclass[10pt]{amsart}
\usepackage{xcolor}
\usepackage{amsmath,amsthm,amssymb,latexsym,soul,cite,mathrsfs}
\usepackage{enumitem}
\usepackage{color,enumitem,graphicx}
\usepackage[colorlinks=true,urlcolor=blue,
citecolor=red,linkcolor=blue,linktocpage,pdfpagelabels,
bookmarksnumbered,bookmarksopen]{hyperref}
\usepackage[english]{babel}

\usepackage[left=2.7cm,right=2.7cm,top=2.5cm,bottom=2.5cm]{geometry}

\numberwithin{equation}{section}

\theoremstyle{plain}
\newtheorem{theorem}{Theorem}[section]
\newtheorem{theoremletter}{Theorem}

\newtheorem{lemma}[theorem]{Lemma}

\newtheorem{corollary}[theorem]{Corollary}
\newtheorem{proposition}[theorem]{Proposition}

\theoremstyle{remark}
\newtheorem{example}[theorem]{Example}
\newtheorem{remark}[theorem]{Remark}
\theoremstyle{definition}

\newcommand{\dx}{\,\mathrm{d}x}

\newcommand{\dxi}{\,\mathrm{d}\xi}
\newcommand{\dtau}{\,\mathrm{d}\tau}
\newcommand{\dxdy}{\,\mathrm{d}x\mathrm{d}y}

\newcommand{\vertiii}[1]{{\left\vert\kern-0.25ex\left\vert\kern-0.25ex\left\vert #1 
    \right\vert\kern-0.25ex\right\vert\kern-0.25ex\right\vert}}

\newcommand{\dive}{\mathrm{div}}
\newcommand{\rad}{\mathrm{rad}}
\DeclareMathOperator{\supp}{supp}

\DeclareMathOperator{\sgn}{sgn}
\newcommand{\loca}{\operatorname{loc}}

\makeatletter
\newcommand{\leqnomode}{\tagsleft@true}
\newcommand{\reqnomode}{\tagsleft@false}
\makeatother
\usepackage[norefs,nocites]{refcheck}
\title[Fractional Schr\"{o}dinger equations]{Concentration-compactness at the mountain pass level for nonlocal Schr\"{o}dinger equations}

\author[J.M.\ do \'O]{Jo\~ao Marcos do \'O}
\author[D.~Ferraz]{Diego Ferraz}

\address[J.M. do \'O]{Department of Mathematics,
Federal University of Para\'{\i}ba
\newline\indent
58051-900, Jo\~ao Pessoa-PB, Brazil}
\email{\href{mailto:jmbo@pq.cnpq.br}{jmbo@pq.cnpq.br}}

\address[D.~Ferraz]{Department of Mathematics,
	Federal University of Para\'{\i}ba
	\newline\indent
	58051-900, Jo\~ao Pessoa-PB, Brazil}
\email{\href{mailto:diego.ferraz.br@gmail.com}{diego.ferraz.br@gmail.com}}


\begin{document}

\begin{abstract}
The current paper is dedicated to the theory of concentration compactness principles for inhomogeneous fractional Sobolev spaces.
This subject for the local case has been studied since the publication of celebrated works due to P.-L. Lions, which laid the broad foundations of the method and outlined a wide scope of its applications. Our study is based on the analysis of the profile decomposition for the weak convergence following the approach of dislocation spaces, introduced by K. Tintarev and K.-H Fieseler. As an application we obtain existence of nontrivial and nonnegative solutions and ground states of fractional Schr\"{o}dinger equations for a wide class of possible singular potentials, not necessarily bounded away from zero. We consider possible oscillatory nonlinearities for both cases, subcritical and critical which are superlinear at origin, without the classical Ambrosetti and Rabinowitz growth condition. In some of our results we prove existence of solutions by means of compactness of Palais-Smale sequences of the associated functional at the mountain pass level. To this end we study and provide the behavior of the weak profile decomposition convergence under the related functionals.  Moreover, we use a Pohozaev type identity in our argument to compare the minimax levels of the energy functional with the ones of the associated limit problem. Motivated by this fact, in our work we also proved that this kind of identities hold for a larger class of potentials and nonlinearities for the fractional framework.
\end{abstract}
\maketitle



\section{Introduction}
The main goal of the present work is to analyze concentration-compactness principles for inhomogeneous fractional Sobolev spaces. As an application we address questions on the existence of solutions for the following nonlocal Schr\"{o}dinger equation
\begin{equation}\label{P}
(-\Delta)^{s} u +a(x)u= f(x,u)\quad\text{in }\mathbb{R}^N,
\tag{$\mathcal{P}_{s}$}
\end{equation}
where 
$0<s<1$ and $(-\Delta)^{s}$ is the fractional Laplacian (see \cite{silvestre,hitchhiker} for more details).

During the past years there has been a considerable amount of research on nonlinear elliptic problems involving fractional Laplacian motivated from the fact this class of problems arise naturally in several branches of mathematical physics. For instance, solutions of \eqref{P}, can be seen as stationary states (corresponding to solitary waves) of nonlinear Schr\"{o}dinger equations of the form $i \phi _t - (- \Delta) ^s \phi +a(x)\phi +f(x,\phi)=0\quad\text{in } \mathbb{R}^N.$ For more motivation we refer to \cite{levy1}.

This paper is motivated by recent advances in the study of existence of solutions for nonlinear and nonlocal Schr\"{o}dinger field equations. In \cite{secchi} S. Secchi investigated the existence of ground state solutions for fractional Schr\"{o}dinger equations by using a minimization argument on the Nehari manifold. He proved existence results under suitable assumptions on the behavior of the potential $a(x)$ and superlinear growth conditions on the nonlinearity. See also \cite{feng}, where B. Feng proved the existence of ground state solutions of \eqref{P}, for the particular case $f(x,t) = |t|^{p-2}t,$ $2<p<2(N+2s)/N ,$ $N \geq 2,$ by using the P.-L. Lions concentration-compactness principle (see \cite{lionscompcase2}). R. Lehrer et al. \cite{lehrer} studied the existence of solutions through projection over an appropriated Pohozaev manifold, assuming that $f(x,t) = a(x)f_0(t),$ where $f_0(t)$ is asymptotically linear that is, $ \lim_{t\rightarrow \infty}f_0(t) / t =1$ and $\lim_{|x|\rightarrow \infty} a(x) = a_\infty>0.$ For the local case ($s=1$), R. de Marchi \cite{reinaldo} studied existence of nontrivial solutions for \eqref{P} assuming that $a(x)$ and $f(x,t)$ are asymptotically $\mathbb{Z}^N$--periodic combining variational methods and concentration-compactness principle. 
In \cite{reinaldo} it was also proved the existence of ground states,  without assuming that $t \mapsto f(x,t)t^{-1}$ is an increasing function. By using a similar approach H. Zhang et al \cite{hzhang} studied existence of ground states and infinitely many geometrically distinct solutions of Eq. \eqref{P}, based on the method of Nehari manifold and Lusternik-Schnirelmann category theory. For recent works on nonlinear Schr\"{o}dinger equations where the Ambrosetti-Rabinowitz condition is not required we cite \cite{lehrer,hzhang,reinaldo}. See also the recent work due to A. Ambrosetti et al. \cite{malchiodi_GS}, where the existence of ground states with potentials vanishing at infinity $a(x)$ and $K(x),$ where $ f(x,t) = K(x) t^{p-1}$ and $2<p<2N/(N-2),$ was studied.

Problems involving potentials bounded away from zero and \textit{critical Sobolev exponent}, precisely, when $f(x,t)=g(x,t)+|t|^{2_s^\ast-2}t,$ $2_s^\ast = 2N/(N-2s),$ where $g(x,t)$ have subcritical growth, we may refer to \cite{shang, shang_gs, liding} and the references given there. In these works, it was crucial the presence of perturbation $g(x,t)$ of the critical power $|t|^{2_s^\ast-2}t.$ Moreover, it was assumed the following condition on the potential $0<\inf_{x \in \mathbb{R}^N} a(x) < \liminf _{|x| \rightarrow \infty} a (x)$ which was introduced by P.L. Rabinowitz in \cite{Rabi1992} to study the local case of Eq. \eqref{P} (see also for the critical case \cite{miyagaki}). We cite \cite{tintahardy, deng,smets} for works on local Schr\"{o}dinger equations with nonlinearities of the pure critical power type (without subcritical perturbation term) and inverse square type potentials. For the fractional case we cite \cite{peral}, where it was studied the existence and qualitative properties of positive solutions.

Motivated by the above works, we study existence of nontrival solutions for Eq. \eqref{P} in several cases, which were not considered by the aforementioned papers. Our potential $a(x)$ may change sign, can have singular points of blow up and even vanish, and the nonlinearity can be considered with subcritical or critical oscillatory growth. We prove some of our existence results by means of compactness of Palais-Smale sequences (for short PS sequences) of the associated functional at the mountain pass level.

In the subcritical case we assume a condition on $a(x)$ which ensures the continuous embedding of the associated space of functions similar to \cite{sirakov}. Nevertheless differently from \cite{sirakov}, we do not impose assumption on $a(x)$ to guarantee the compactness of the Sobolev embedding. To compensate, we ask that the limit of $a(x),$ as $|x|$ goes to infinity, exists and is positive, or alternatively, that $a(x)$ is $\mathbb{Z}^N$--periodic. Moreover, by considering similar assumptions made in \cite{manasses}, the potential does not need to be bounded from below by a constant. We also take account the case where the nonlinearity has oscillatory behavior and does not satisfy the typical assumption of Ambrosetti-Rabinowitz. Similar to \cite{reinaldo}, the nonlinearity $f(x,t)$ is supposed to has a periodic asymptote $f_\mathcal{P}(x,t)$, which allows us to ``transfer'' the usual assumptions to it. Also we mention that we complement and improve some results in \cite{reinaldo}, since we consider the fractional case and we do not require the monotonicity of $t \mapsto f_\mathcal{P}(x,t)t^{-1}.$


In the critical case, inspired in some ideas contained in \cite{tintahardy}, we treated in this work a class of potentials somehow different, since we consider a general class that include as a particular case the inverse fractional square potential $a(x)=-\lambda |x|^{-2s},$ where $0<\lambda < \Lambda _{N,s}$ and $\Lambda _{N,s}$ is the sharp constant of the Hardy-Sobolev inequality,
\begin{equation}\label{par_hardy}
\Lambda _{N,s} \int _{\mathbb{R}^N}|x|^{-2s}u^2 \dx \leq \int _{\mathbb{R}^N} |\xi |^{2s} \left |\mathscr{F} u \right| ^2 \dxi,\quad\forall \, u \in C^\infty _0 (\mathbb{R}^N),
\end{equation}
where $\Lambda _{N,s}$ is the sharp constant of this inequality. Further details about \eqref{par_hardy} can be found in \cite{hardysharp}. 
Here we consider self-similar nonlinearities which generalize the idea of oscillations about the critical power $|t|^{2_s ^\ast-2}t,$ turning the approach by variational methods more involved. This class of functions were introduced in \cite{paper1} and for the local case in \cite{tinta_pos,tintabook,tintapaper}.

We are able to avoid the monotonicity of $t \mapsto f_\mathcal{P}(x,t)t^{-1}$ by comparing the mini\-max level of the associated energy functional of Eq. \eqref{P} with the one of the associated limit problem. To this end we use a Pohozaev type identity and an appropriated concentration-compactness principle. The proof of this identity is essentially based in the use of the so called $s$-harmonic extension introduced by L. Caffarelli and L. Silvestre \cite{caf_silv} and remarks contained in \cite{moustapha} and \cite{frac_niremberg}, which allow us to ``transform'' the nonlocal problem \eqref{P} in a local one, from this we may apply a truncation argument. Our method of proof is more general than the usual one, in the sense that in our argument we do not have to study the behavior of solutions in the whole space $\mathbb{R}^N;$ and we also can consider singular potentials (see Proposition \ref{pohozaev_id}).

It is worth to mention that the main difficulty to approach problem \eqref{P} using variational methods lies on the lack of compactness, which roughly speaking, originates from the invariance of $\mathbb{R}^N$ with respect to translation and dilation and, analytically, appears because of the non-compactness of the Sobolev embedding. We are able to overcome this difficulty by relying in a concentration-compactness principle by means of profile decomposition for weak convergence in inhomogeneous fractional Sobolev spaces, which can be considered as extensions of the Banach-Alaoglu theorem (see Theorem \ref{teo_tinta_frac_sub}). This kind of results were considered in various settings, for instance we may cite the ones in \cite{gerard,struwe,palatucci}. It describes how the convergence of a bounded sequence fails in the considered space. Our approach in this matter was motivated by \cite{tintacocompact} and based in the abstract version of profile decomposition in Hilbert spaces due to K. Tintarev and K.-H Fieseler \cite{tintabook}. It seems for us that this approach is more appropriated to study existence of nontrival solutions for problems like \eqref{P}, under our settings, then the usuals ones using P.-L. Lions concentration-compactness principle (see \cite[Lemma 2.2]{felmer}).

Another important goal here is the study of the existence of ground states for $\eqref{P},$ i.e., nontrivial solutions with least possible energy. We consider three cases: First when \eqref{P} is invariant under the action of translations in $\mathbb{Z}^N$ (subcritical growth), second when \eqref{P} is invariant under dilations (critical growth), and third when the monotonicity of $t \mapsto f(x,t)t^{-1}$ is considered.

This paper is organized as follows. Sect.~\ref{s_profile_decomp} is we have the description of the profile decomposition of bounded sequences in $H^s(\mathbb{R}^N)$. In Sect.~\ref{s_Application}, we give some applications of the profile decomposition to study the existence of mountain-pass type solutions of \eqref{P} for the autonomous and non-autonomous cases. In Sect.~\ref{s_Preliminaries}, we state some basic results (without prove) on the fractional Sobolev spaces. In Sect.~\ref{s_proof_profile_decomp}, we prove the abstract result stated in Sect. \ref{s_profile_decomp}. In Sect.~\ref{s_variational}, we provide a suitable variational settings to prove our main results. More precisely, we describe the limit under the profile decomposition of the PS sequences at the mountain pass level of the Lagrangian of \eqref{P}. We also prove that solutions for \eqref{P} in the autonomous case satisfy a Pohozaev type identity. Sections \ref{s_proof_GS}--\ref{s_proof_crit} are dedicated to the proof of our main results concerning the existence of mountain pass solutions of Eq. \eqref{P}.


\section{Profile Decomposition for weak convergence in fractional Sobolev spaces}\label{s_profile_decomp}
Assume $0<s<N/2$ and let $\mathcal{D}^{s,2}(\mathbb{R}^N)$ be the homogeneous fractional Sobolev space, which are defined as the completion of $C ^{\infty } _0 (\mathbb{R}^N)$ under the norm $[u]_s^2 :=  \int _{\mathbb{R}^N} |\xi|^{2s} \left |\mathscr{F} u \right| ^2 \dxi.$ It is well known that $\mathcal{D}^{s,2}(\mathbb{R}^N)$ is continuous embedded in the space $L^{2^\ast_s}(\mathbb{R}^N).$
The following result represents the results developed by K. Tintarev and K.-H Fieseler in \cite{tintabook} to obtain concentration-compactness by means of profile decomposition for bounded sequences in abstract Hilbert spaces, and was studied in \cite{paper1,palatucci} for the fractional framework.
\begin{theoremletter}\cite[Theorem 2.1]{paper1}\label{teo_tinta_frac}
	Let $(u_k) \subset \mathcal{D}^{s,2}(\mathbb{R}^N)$ be a bounded sequence, $\gamma > 1$ and $0<s<\min\{1,N/2\}.$ Then there exist $\mathbb{N}_{\ast } \subset \mathbb{N},$ disjoints sets (if non-empty) $\mathbb{N} _{0}, \mathbb{N} _{- }, \mathbb{N}_ {+} \subset \mathbb{N}, $ with $\mathbb{N}_{\ast } = \mathbb{N} _{0} \cup  \mathbb{N}_ {+} \cup \mathbb{N} _{-} $ and sequences $(w ^{(n)}) _{n \in \mathbb{N}_{\ast } } \subset \mathcal{D}^{s,2}(\mathbb{R}^N),$ $(y_k ^{(n)}) _ {k \in \mathbb{N}} \subset \mathbb{Z}^N,$ $(j _k ^{(n)}) _{k \in \mathbb{N}}  \subset \mathbb{Z},$ $n \in \mathbb{N}_{\ast },$ such that, up to a subsequence of $(u_k),$
	\begin{align}
	\qquad \qquad &\gamma ^{-\frac{N-2s}{2}j_k ^{(n)}}u_k \big(\gamma ^{- j_k ^{(n)}} \cdot + y_k^{(n)} \big)\rightharpoonup w^{(n)} \text{ as } k \rightarrow \infty,\mbox{ in } \mathcal{D}^{s,2}(\mathbb{R}^N),&\nonumber \\ 
	\qquad \qquad &\big|j_k ^{(n)} - j_k ^{(m)}\big|+\big|\gamma ^ {j_k ^{(n)}}( y^{(n)}_k - y^{(m)}_k )\big| \rightarrow \infty , \text{ as } k \rightarrow \infty,\text{ for } m \neq n,&\nonumber \\ 
	\qquad \qquad &\sum _{n \in \mathbb{N}_{\ast}}\big \| w^{(n)} \big \|^2  \leq \limsup_k \|u_k \| ^2,&\nonumber \\ 
	\qquad \qquad & u_k - \sum _{n \in \mathbb{N}_{\ast }} \gamma ^{\frac{N-2s}{2} j_k ^{(n)} } w^{(n)}\big(\gamma^{j _k ^{(n)}} ( \cdot - y_k^{(n)} ) \big) \rightarrow  0, \text{ as } k \rightarrow \infty, \text{ in } L^{2^{\ast } _s}(\mathbb{R}^{N}),&\label{seis.quatro}
	\end{align}
	and the series in \eqref{seis.quatro} converges uniformly in $k.$ Furthermore, $1 \in \mathbb{N} _0,$ $y_k ^{(1)} = 0;$ $j _k ^{(n)} = 0 $ whenever $n \in \mathbb{N}_0;$ $j_k ^{(n)} \rightarrow -\infty$ whenever $n \in \mathbb{N}_{- };$ and $j_k ^{(n)} \rightarrow +\infty$ whenever $n \in \mathbb{N}_{ +}.$
\end{theoremletter}
In \cite{gerard} P. G\'{e}rard introduced this subject in a fractional framework. Theorem \ref{teo_tinta_frac} can be used to prove the fractional version of Lions concentration-compactness principle due to G.~Palatucci and A. Pisante  \cite[Theorem 5]{palatucci}. Indeed, Theorem \ref{teo_tinta_frac} improves \cite[Theorem 5]{palatucci} for the case $\Omega = \mathbb{R}^N,$ since the sums of Dirac masses that appear in this result comes from the profiles given in \eqref{seis.quatro}. The new notion of criticality introduced in \cite{paper1} together with the concentration-compactness given in Theorem \ref{teo_tinta_frac} can lead to a new way to approach nonlocal elliptic problems involving critical growth. For instance, it is usual to apply \cite[Theorem 5]{palatucci} to study Eq. \eqref{P} considering nonlinearities of the type $f(x,t) = K(x) |t|^{2_s^\ast - 2} t.$ With aid of Theorem \ref{teo_tinta_frac}, it is possible to consider more general self-similar critical nonlinearities (for more details see \cite[Sect. 3.1]{paper1}).

\begin{remark}\label{r_radial_Dcompact}
We can consider the closed subspace consisting of radial functions $\mathcal{D}^{s,2}_\rad (\mathbb{R}^N) = \left\lbrace u \in \mathcal{D}^{s,2}(\mathbb{R}^N) : u(x) = u(y) \mbox{ if } |x|=|y|\right\rbrace$ to get more compactness.
In this case, by the same argument of \cite[Proposition 5.1]{tintabook}, we have $w^{(n)} \in \mathcal{D}^{s,2}_\rad (\mathbb{R}^N)$ with $w^{(n)} = 0,$ for all $n \in \mathbb{N}_0.$
\end{remark}

In this paper, we prove the inhomogeneous case of Theorem \ref{teo_tinta_frac}, that is, for the space $H^s (\mathbb{R}^N) = \left\lbrace  u \in L^2(\mathbb{R}^N) : | \cdot|^{2s} \mathscr{F}u \in L^2(\mathbb{R}^N) \right\rbrace ,\quad 0<s\leq N/2,$ with the norm $\|u\|^2 := \int _{\mathbb{R}^N} |\xi|^{2s} \left |\mathscr{F} u \right| ^2 +u^2\dxi.$ It is known that $H^s (\mathbb{R}^N)$ is continuous embedded in $L^p(\mathbb{R}^N),$ for $2\leq p\leq 2_s^\ast,$ in the case where $N>2s,$ and in $L^p(\mathbb{R}^N),$ for $2\leq p< \infty,$ in the case where $N=2s.$ The following version of Theorem \ref{teo_tinta_frac} will be used to study the existence of solutions of \eqref{P} for the case where $f(x,t)$ has subcritical growth. Next we set $2_s^\ast = \infty,$ when $N=2s.$
\begin{theorem}\label{teo_tinta_frac_sub}
Let $(u_k) \subset H^s(\mathbb{R}^N)$ be a bounded sequence and $0<s\leq N/2.$ Then there exist $\mathbb{N}_{0 } \subset \mathbb{N},$ $(w ^{(n)}) _{n \in \mathbb{N}_{0 } }\subset H^s(\mathbb{R}^N),$ $(y_k ^{(n)}) _ {k \in \mathbb{N}}\subset\mathbb{Z}^N,$ $n \in \mathbb{N}_{0 },$ such that, up to a subsequence of $(u_k)$
\begin{align}
\qquad \qquad &u_k (\cdot + y_k^{(n)} )\rightharpoonup w^{(n)}, \text{ as } k \rightarrow \infty,\mbox{ in } H^s(\mathbb{R}^N),&\label{seis.um.sub}\\
\qquad \qquad &|y^{(n)}_k - y^{(m)}_k| \rightarrow \infty , \text{ as } k \rightarrow \infty,\text{ for } m \neq n,&\label{seis.dois.sub} \\
\qquad \qquad &\sum _{n \in \mathbb{N}_0}\| w^{(n)} \|^2 \leq \limsup_k \|u_k \| ^2,&\label{seis.tres.sub} \\
\qquad \qquad &u_k - \sum _{n \in \mathbb{N}_0} w^{(n)}(\cdot + y_k^{(n)} ) \rightarrow  0, \text{ as } k \rightarrow \infty, \text{ in } L^{p}(\mathbb{R}^{N}),&\label{seis.quatro.sub}
\end{align}
for any $p \in (2,2^{\ast } _s)$. Moreover, $1 \in \mathbb{N}_0,\ y_k ^{(1)} = 0$ and the series in \eqref{seis.quatro.sub} converges uniformly in $k.$
\end{theorem}

Those profile decompositions for bounded sequence are unique up to a permutation of index, and constant operator (see \cite[Proposition 3.4]{tintabook}). Theorem \ref{teo_tinta_frac_sub} is the fractional counterpart of \cite[Corollary 3.3]{tintabook} and it describes how bounded sequences in $H^s (\mathbb{R}^N)$ fails to converges in $L^p(\mathbb{R}^N),\ 2<p<2_s^\ast.$ This ``error'' of convergence is generated by the invariance of action of translations in $H^s (\mathbb{R}^N).$ Moreover, it can be seen as an alternative result to a version of Lions Lemma of compactness, for $H^s (\mathbb{R}^N),$ proved in \cite[Lemma 2.2]{felmer}. Also, we empathize that the profile decomposition of Theorem \ref{teo_tinta_frac_sub} is given by translations of the form $u \mapsto u(\cdot - y),$ $y\in \mathbb{Z}^N$ and, differently from \cite[Theorem 4 and Theorem 8]{palatucci}, we also consider the limit case where $s=N/2.$ Theorem \ref{teo_tinta_frac_sub} holds thanks to a cocompactness result contained in \cite{tintacocompact}, and using the abstract approach of \cite{tintabook}, considering $H^s(\mathbb{R}^N),$ $0<s\leq N/2,$ as the starting Hilbert space.

\section{Nonlinear fractional Schr\"{o}dinger equation}\label{s_Application}
\subsection{Hypothesis}
In order to describe our results in a more precisely way, next we state the main assumptions on the potential $a(x)$ and the nonlinearity $f(x,t).$ We always assume that $0<s<\min\{1,N/2 \}.$ We denote by $\|\cdot\|_p$ and $\|\cdot\|_\infty$ the norms of the spaces $L^p(\mathbb{R}^N),$ $1\leq p < \infty$ and $L ^\infty(\mathbb{R}^N)$ respectively. $|A|$ denotes the Lebesgue measure of the set $A \subset \mathbb{R}^N.$
\subsubsection{Subcritical case} Let us first introduce the assumptions on $a(x)=V(x)-b(x).$
\begin{flalign}\tag{$V_1$}\label{V_pe}
V(x) \text{ is a }\mathbb{Z}^N\text{--periodic function in the space }L ^{\sigma} _{\loca } (\mathbb{R}^N) \text{ for some }\sigma > 2N/(N+2s).&&
\end{flalign}
\begin{flalign}\tag{$V_2$}\label{V_sirakov}
\begin{aligned}
&\exists \ \mathcal{B}>0 \ \text{such that (s.t.) }V(x) \geq -\mathcal{B}\ \text{ almost everywhere (a.e.) }\ x \in \mathbb{R}^N \text{ and }\\
&\mathcal{C}_{V} := \inf _{ u \in C ^\infty _0 ( \mathbb{R} ^N),\|u\|_2 = 1 } \int _{\mathbb{R}^N} |(-\Delta)^{s/2} u | ^2 +V(x)u^2\dx>0.
\end{aligned}
&&
\end{flalign}
\begin{flalign}\tag{$V_3$}\label{V_be}
\begin{aligned}
&\exists \ \beta> N/2s \text{ s.t. } 0\leq b(x)\in L^{\beta}(\mathbb{R}^N)\text{ and for }\beta' = \beta/(\beta -1),\\
&\|b(x)\|_{\beta}<\mathcal{C}_{V} ^{(\beta)}:= \inf _{u \in H^s_{V}(\mathbb{R}^N),\|u\|_{2\beta'}=1} \int _{\mathbb{R}^N} |(-\Delta)^{s/2} u | ^2 +V(x)u^2\dx.
\end{aligned}
&&
\end{flalign}
\begin{flalign}\tag{$V_4$}\label{V_as}
\exists \ \sigma > N/2s \text{ s.t } V(x) \in L ^{\sigma} _{\loca } (\mathbb{R}^N)\text{ and } \exists\ V _{\infty}:=\lim _{|x|\rightarrow \infty} V(x) > 0.&&
\end{flalign}
We also assume the following conditions on the nonlinear function $f(x,t).$

\begin{flalign}\tag{$f_1$}\label{bem_def}
\begin{aligned}
& f: \mathbb{R}^N \times \mathbb{R} \rightarrow \mathbb{R}
  \text{ is a Carath\'{e}odory function and }\forall \, \varepsilon >0,\ \exists \, C_\varepsilon >0, \ p_\varepsilon \in (2,2_s ^\ast)\text{ s.t. }\\
&|f(x,t)| \leq \varepsilon (|t| + |t| ^{2 _s ^{\ast} -1}) + C_{\varepsilon} |t| ^{p_{\varepsilon } -1},\text{ a.e. } x \in \mathbb{R}^N, \ \forall \,  t \in \mathbb{R}.\nonumber
\end{aligned}
&&
\end{flalign}
\begin{flalign}\tag{$f_2$}\label{A-R}
\exists \ \mu > 2 \quad \text{s.t.}\quad \mu F(x,t):= \mu \int _0 ^t f(x, \tau)\dtau \leq f(x,t)t,\text{ a.e. } x \in \mathbb{R}^N, \  \forall \, t \in \mathbb{R} &&
\end{flalign}
\begin{flalign}\tag{$f_3$}\label{posi_algum}
\begin{aligned}
&\exists \ R>0,\ t_0 > 0,\ x_0 \in \mathbb{R}^N\text{ s.t. setting } C_R(x_0,t_0) = (B_{R+1}(x_0)\setminus B_R(x_0)) \times [0,t_0],\\
&|B_R|\inf _{B_R(x_0)} F(x,t_0) + |B_{R+1}\setminus B_R| \inf_{(x,t) \in  C_R(x_0,t_0)} F(x,t) >0.
\end{aligned}
 &&
\end{flalign}
\begin{flalign}\tag{$f_4$}\label{quadratic}
\begin{aligned}
&\lim_{t \rightarrow 0} \frac{f(x,t)}{t}=0\quad\text{and}\quad\lim _{|t|\rightarrow \infty}\frac{F(x,t)}{t^2} = \infty,\text{ uniformly in }x\text{ and }\\
&\forall \ \text{compat } K \subset \mathbb{R}, \ \exists C=C(K)>0 \ \text{s.t.} \ |f(x,t)|\leq C, \text{ a.e }x \in \mathbb{R}^N ,\ \forall \, t \in K.
\end{aligned}
&&
\end{flalign}
\begin{flalign}\tag{$f_5$}\label{dife_posi}
\forall \, 0<a<b,\ \inf_{x \in \mathbb{R}^N} \inf _{a \leq |t| \leq b} \mathcal{F}(x,t) >0, \text{ where }\mathcal{F}(x,t):=\frac{1}{2}f(x,t)t-F(x,t). &&
\end{flalign}
\begin{flalign}\tag{$f_6$}\label{dife_esti}
\exists \ p_0 > \max\{1,N/2s\},\, a_0,\ R_0 >0 \text{ s.t. } |f(x,t)|^{p_0} \leq a_0 |t|^{p_0} \mathcal{F}(x,t),\text{ a.e.}x \in \mathbb{R}^N,\ \forall \, |t|>R_0. &&
\end{flalign}
\begin{flalign}\tag{$f_7$}\label{lim_infinito}
\begin{aligned}
&\exists\ \mathbb{Z}^N\text{--periodic function }f_{\mathcal{P}}(x,t), \text{ satisfyng \eqref{bem_def} and either \eqref{A-R}--\eqref{posi_algum} or \eqref{quadratic}. } \ s.t.\\
&\lim _{|x| \rightarrow \infty}|f(x,t) - f_{\mathcal{P}}(x,t)|=0,\text{ uniformly in compact subsets of }\mathbb{R}.
\end{aligned}
&&
\end{flalign}
\begin{flalign}\tag{$f_8$}\label{increasing}
\text{For a.e. }x \in \mathbb{R}^N,\text{ the function } t\mapsto f_\mathcal{P} (x,t)/|t|,\text{ is strict increasing in }\mathbb{R}. &&
\end{flalign}
Next we assume that $f_{\mathcal{P}}(x,t)$ is independent of $x$ and we denote by $f_\infty(t) = f_\mathcal{P}(t).$
\begin{flalign}\tag{$f_9$}\label{lim_infinito_infty}
f_\infty (t) \in C^1(\mathbb{R}),\ \exists \, t_0 >0\text{ s.t. }F_\infty(t_0) -\frac{V_\infty}{2}t^2_0 > 0,\ \text{where } F_\infty (t) = \int_0 ^t f_\infty(\tau ) d \tau. &&
\end{flalign}
We look for solutions in the Hilbert space $H^s_{V} (\mathbb{R}^N)$ defined as the completion of $C _0 ^\infty (\mathbb{R}^N)$ with respect to the norm and scalar product
\begin{equation*}
\|u\|_{V} ^2 := \int _{\mathbb{R}^N} |(-\Delta)^{s/2} u | ^2 +V(x)u^2\dx\quad \text{and}\quad (u,v)_{V} := \int _{\mathbb{R}^N} (-\Delta )^{s/2} u (-\Delta )^{s/2} v +V(x) u v \dx,
\end{equation*}
see Proposition \ref{welldef}. Writing, $a(x) = V(x) - b(x),$ and assuming \eqref{V_be} and \eqref{bem_def} we can see that the functional associated with \eqref{P}, $I :H_{V}^s (\mathbb{R}^N ) \rightarrow \mathbb{R}$ given by
\begin{equation*}
I(u) = \frac{1}{2}\|u\|_{V} ^2 -\frac{1}{2}\int _{\mathbb{R}^N}b(x) u^2 \dx -  \int _{\mathbb{R}^N} F(x,u) \dx,
\end{equation*}
is well defined, belongs to $C^1 (H _{V} ^s (\mathbb{R}^N ))$, with
\begin{equation*}
I'(u)\cdot v= \int _{\mathbb{R}^N} (-\Delta )^{s/2} u (-\Delta )^{s/2 } v +(V(x) -b(x)) u v \dx - \int _{\mathbb{R}^N} f(x,u) v \dx,\quad u, v \in H^{s,2} _{V}(\mathbb{R}^N).
\end{equation*}
Thus critical points of $I$ correspond to weak solutions of \eqref{P} and conversely. Consider the minimax level
\begin{equation}\label{minimax}
c(I)= \inf_{\gamma \in \Gamma_I } \sup_{t \geq 0 } I(\gamma (t)),
\end{equation}
where
\begin{equation}\label{minimax_paths}
\Gamma _{I} = \left\lbrace \gamma \in C([0, \infty ),  H^s _{V}(\mathbb{R} ^N) ) : \gamma(0)=0, \ \lim _{t \rightarrow \infty}I(\gamma (t)) = - \infty \right\rbrace.
\end{equation}
Associated with the limit functions given in \eqref{V_as}, \eqref{lim_infinito}, \eqref{lim_infinito_infty}, we consider the $C^1$ functionals
\begin{align*}
I _{\mathcal{P}} (u)&:= \frac{1}{2}\| u \| _{V} ^2 - \int _{\mathbb{R}^N} F_{\mathcal{P}}(x,u) \dx,\quad u \in H^s _{V}(\mathbb{R}^N),\\
I _\infty (u) &:=\frac{1}{2} \|u\|_{V_\infty} ^2 - \int _{\mathbb{R}^N} F_{\infty}(u) \dx,\quad u \in H^s _{V} (\mathbb{R}^N),
\end{align*}
where $F_\mathcal{P}(x,t) = \int_0 ^t f _\mathcal{P} (x, \tau) d \tau.$ Similarly, as in \eqref{minimax} and \eqref{minimax_paths}, we can define $c(I_\mathcal{P}),$ $c(I_\infty), \Gamma _{I_\mathcal{P}}$ and $\Gamma _{I_\infty}.$ Next we state the assumption on the minimax levels to guarantees compactness of the PS sequences at the mountain pass level of $I.$
\begin{flalign}\tag{$f_{10}$}\label{minimax_comper}
c(I) < c(I _{\mathcal{P}}).&&
\end{flalign}
\begin{flalign}\tag{$f'_{10}$}\label{minimax_comper'}
c(I) < c(I _\infty).&&
\end{flalign}
\noindent In the autonomous case, $f(x,t)=f(t),$ we consider the following variant of \eqref{posi_algum}.
\begin{flalign}\tag{$f_3'$}\label{posi_algum_auto}
\exists \ t_0 > 0 \text{ s.t. }F(t_0) >0. &&
\end{flalign}
\subsubsection{Critical case} Next we state our hypothesis on $a(x)\equiv V(x),$ assuming that $b(x)\equiv 0.$

\begin{flalign}\tag{$V_1 ^\ast$}\label{V_sirakov_ast}
\begin{aligned}
& \exists \text{ a finte set }  \mathcal{O} \subset \mathbb{R}^N \text{  s.t. }  V(x) \in L ^1_{\loca } (\mathbb{R}^N)\cap C(\mathbb{R}^N \setminus \mathcal{O}), \; V(x) \leq 0 \text{ a.e. in } \mathbb{R}^N \text{ and } \\
&  \mathcal{C}^\ast_{V}:= \inf _{ u \in C ^\infty _0 ( \mathbb{R} ^N)\setminus\{0\} } \frac {\int _{\mathbb{R}^N} |(-\Delta)^{s/2} u | ^2 +V(x)u^2\dx}{\int _{\mathbb{R}^N} |V(x)| u^2 \dx}>0.
\end{aligned}
&&
\end{flalign}

\begin{flalign}\tag{$V_2 ^\ast$}\label{V_limits_ast}
\begin{aligned}
&\exists \, a_\ast \in \mathbb{R}^N\text{ s.t.  there exist the limits uniformly on every compact subset of } \mathbb{R}^N,\\
& V_+ (x) := \lim_{\lambda \rightarrow \infty } \lambda ^{-2s} V(\lambda ^{-1} (x+a_\ast)) \text{ and } 
 V_{-} (x) :=  \lim_{\lambda \rightarrow 0 } \lambda ^{-2s} V(\lambda ^{-1} (x+a_\ast)). \\
 &\text{Moreover }V_\kappa (x) \text{ satisfies \eqref{V_sirakov_ast} if }V_\kappa (x) \not \equiv 0\text { for }\kappa =+,-.\text{ Also, }\lim _{|x|\rightarrow \infty} V(x)=0.\\
\end{aligned}
&&
\end{flalign}


\begin{flalign}\tag{$V_3 ^\ast$}\label{V_conve_seq_crit}
\begin{aligned}
&\forall \, (\lambda_k)\subset \mathbb{R}^+\text{ s.t. either }|\lambda_k|\rightarrow\infty\text{ or }|\lambda_k|\rightarrow 0;\text{ and }(y_k) \subset \mathbb{R}^N\text{ s.t. }|\lambda_k y_k|\rightarrow \infty, \\
& \exists \lim _{k \rightarrow \infty} \lambda_k ^{-2s} V(\lambda _k ^{-1}x + y_k ) = 0,\text{ uniformly on every compact subset of } \mathbb{R}^N 
\end{aligned}
&&
\end{flalign}
We assume the following conditions on the nonlinearity $f(x,t).$
\begin{flalign}\tag{$f_1 ^\ast$}\label{bem_def_ast}
\begin{aligned}
&f: \mathbb{R}^N \times \mathbb{R} \rightarrow \mathbb{R}
  \text{ is a Carath\'{e}odory function satisfying the growth condition,}\\
&\exists \ C>0,\text{ s.t. }|f(x,t)| \leq C |t| ^ {2 ^\ast _s - 1}\text{ a.e. } x \in \mathbb{R}^N,\ \forall \, t \in \mathbb{R}.
\end{aligned}
&&
\end{flalign}
\begin{flalign}\tag{$f_2 ^\ast$}\label{f_extra}
\begin{aligned}
&\forall \, a_1, \ldots, a_{M}\in \mathbb{R},\ \exists C=C(M)>0 \text{ s.t. }\\
&\left|F\left(x,\sum _{n=1} ^{M} a_n \right) - \sum_{n=1}^{M}F(x,a_n )\right| \leq C(M) \sum _{m \neq n \in \{1,\ldots,M\}} |a_n| ^{2^{\ast} _s -1} |a_m|\ \text{ a.e. } x \in \mathbb{R}^N.
\end{aligned}
&&
\end{flalign}
\begin{flalign}\tag{$f_3 ^\ast$}\label{limites_ast}
\begin{aligned}
&\exists \ \gamma > 1 \text{ s.t. there exist the following limits uniformly on every compact subset of } \mathbb{R}^N,\\
& f_0 (t) := \lim _{|x| \rightarrow \infty} f(x,t), \\
& f_+(t):= \lim _{j \in \mathbb{Z}, j \rightarrow + \infty} \gamma ^{- \frac{N+2s}{2}j} f\left(\gamma^{-j} x , \gamma ^{ \frac{N-2s}{2}j}  t \right),\\
& f_{-}(t):= \lim _{j \in \mathbb{Z}, j \rightarrow - \infty} \gamma ^{- \frac{N+2s}{2}j} f\left(\gamma ^{-j} x, \gamma ^{ \frac{N-2s}{2}j}  t \right),\\
&\text{and the primitive }F_\kappa (t) \text{ satisfies \eqref{posi_algum_auto} for }\kappa =0,+,-.
\end{aligned}
&&
\end{flalign}
\begin{flalign}\tag{$f_4 ^\ast$}\label{increasing_crit}
\text{The function }t\mapsto \frac{f_{\kappa} (t)}{|t|} \text{ is strict increasing for }\kappa=0,+,-.&&
\end{flalign}

From \eqref{V_sirakov_ast} we can see that $\| \cdot \|_{V}$ defines a norm in $\mathcal{D}^{s,2}(\mathbb{R}^N)$ which is equivalent to the standard one (see Proposition \ref{welldef}). Thus, the energy functional $I_\ast :  \mathcal{D}^{s,2}(\mathbb{R}^N) \rightarrow \mathbb{R}$ given by
\begin{equation*}
I_\ast(u) = \frac{1}{2}\|u\|_{V} ^2 -  \int _{\mathbb{R}^N} F(x,u) \dx,\quad u \in \mathcal{D}^{s,2}(\mathbb{R}^N)
\end{equation*}
is well defined and is continuously differentiable provided that \eqref{bem_def_ast} holds. We can define $c(I_\ast)$ and $\Gamma _{I_\ast}$ similarly as in \eqref{minimax} and \eqref{minimax_paths}, by just replacing $H^s _{V} (\mathbb{R} ^N)$ by $\mathcal{D}^{s,2}(\mathbb{R}^N).$ 

We use the next assumptions to compare certain minimax levels. 
\begin{flalign}\tag{$\mathscr{H}^\ast$}\label{suficient_ast}
\begin{aligned}
&V(x)\leq V_\pm(x),\text{ a.e. }x\in \mathbb{R}^N,\\
&F _\kappa (t) \leq  F(x,t),\text{ a.e. } x \in \mathbb{R}^N,\ \forall \, t \in \mathbb{R}\text{ for any }\kappa = 0, +, -.
\end{aligned}
&&
\end{flalign}
\begin{flalign}\tag{$\mathscr{H}^\ast_0$}\label{suficient_ast_0}
\begin{aligned}
&\text{Assume \eqref{suficient_ast}. The first inequality in \eqref{suficient_ast} is strict for a set of positive measure or}\\
&\exists \ \delta >0 \text{ s.t. the second inequality in \eqref{suficient_ast} is strict a.e. } x \in \mathbb{R}^N,  \, \forall \, t \in (-\delta,\delta).
\end{aligned}
&&
\end{flalign}
To study the autonomous case $f(x,t) = f(t)$ we assume that the nonlinearity is self-similar,
\begin{flalign}\tag{$f_5^\ast$}\label{selfsimilar}
\exists \ \gamma>1 \text{ s.t. }F(t)=\gamma ^{-Nj} F\left(\gamma ^{\frac{N-2s}{2}j} t \right),\ \forall \, t \in \mathbb{R},\ j \in \mathbb{Z}.
&&
\end{flalign}

\subsection{Statement of the main existence results}
We first state our results on the existence of ground states for Eq. \eqref{P} for subcritical and critical growth.
\begin{theorem}\label{teoremao}~
\begin{enumerate}[label=(\roman*)]
\item\label{GS_UM} Suppose that $f(x,t)$ and $a(x)\equiv V(x)$ are $\mathbb{Z}^N-$periodic and satisfy \eqref{bem_def}--\eqref{posi_algum} or  \eqref{posi_algum}--\eqref{dife_esti} and \eqref{V_pe}--\eqref{V_sirakov} respectively. Then the equation \eqref{P} has a ground state.
\item\label{GS_DOIS} Suppose that $f(t) \in C^1 (\mathbb{R}^N)$ satisfies \eqref{posi_algum_auto} and \eqref{selfsimilar}. Let $0<\lambda<\Lambda _{N,s}$ given in \eqref{par_hardy} and $\mathcal{G}=\{ u \in \mathcal{D}^{s,2} (\mathbb{R}^N ):\int_{\mathbb{R}^N}F(u)\dx=1 \}.$ Then, there is a radial minimizer $w$ for
	\begin{equation}\label{min_crit}
	\mathcal{I}_\lambda = \inf _{ u \in \mathcal{G}} \int _{\mathbb{R}^N} |(-\Delta)^{s/2} u| ^2 -\lambda |x|^{-2s}u^2 \dx.
	\end{equation}
	Furthermore, for any $w$ minimizer of \eqref{min_crit}, there exists $\alpha >0$ such that $u=w(\cdot / \alpha )$ is a ground state of \eqref{P} for $a(x) = -\lambda |x|^{-2s}.$
\end{enumerate}
\end{theorem}
Theorem \ref{teoremao} take into account the invariance of $I$ under the action of translations and dilations in $H^s(\mathbb{R}^N)$ and $\mathcal{D}^{s,2}(\mathbb{R}^N),$ to obtain the concentration-compactness of Palais-Smale and minimizing sequences in each case respectively. These properties are sufficient to ensure existence of ground states of \eqref{P}. Our results improve and complement \cite{reinaldo} for the fractional framework, we consider potential $a(x)$ and nonlinearity $F(x,t)$ which can change sign. In Theorem \ref{teoremao} (ii),  we do not require the Ambrosetti-Rabinowitz condition \eqref{A-R}. Our argument to prove Theorem \ref{teoremao} (ii) involves a Pohozaev type identity, and as usual we require $C^1$-- regularity of $f(t)$.
\begin{theorem}\label{teo_GS_alt}
Nontrivial weak solutions in $H^s _V (\mathbb{R}^N)$ of \eqref{P} at the mountain pass level are ground states. Precisely, for the Nehari manifold $\mathcal{N} = \left\lbrace u \in H^s _{V} (\mathbb{R}^N) \setminus \{0 \} : I'(u)\cdot u = 0\right\rbrace,$ consider
\begin{equation*}
\bar{c} (I) := \inf _{u \in H^s _{V} (\mathbb{R}^N ) \setminus \{0\} } \sup_{t \geq 0} I (t u)\quad \text{and}\quad c _{\mathcal{N} } (I) := \inf_{u \in \mathcal{N} } I(u).
\end{equation*}
Assume that $V(x)\in L^1 _{\loca}(\mathbb{R}^N),$  $a(x)=V(x) - b(x)$ satisfies \eqref{V_sirakov}--\eqref{V_be} and $f(x,t)$ fulfill \eqref{bem_def}--\eqref{A-R}. Moreover, suppose that
\begin{equation}\label{increasing_total}
t\mapsto \frac{f (x,t)}{|t|}\text{ is strict increasing in }\mathbb{R}, \text{ a.e. }x \in \mathbb{R}^N.
\end{equation}
Then $c(I) = \bar{c} (I) = c _{\mathcal{N} } (I).$ 
\end{theorem}
Theorem \ref{teo_GS_alt} improves some results in \cite{secchi} since we deal with the case where $a(x)$ may changes sign and is not necessarily bounded from below, also with nonlinearity having the behavior at $0$ described by \eqref{bem_def'}. Moreover, Theorem \ref{teo_GS_alt} proves the existence of ground state by replacing the aforementioned invariance by  \eqref{increasing_total}. In fact, our results below give some conditions that guarantee existence of nontrivial weak solutions in $H^s_V(\mathbb{R}^N)$ at the mountain pass level. 

Our next results are on the existence of weak solutions of \eqref{P} at the mountain-pass level by using the concentration-compactnes principle.
\begin{theorem}\label{teo_sub}
Assume that \eqref{bem_def}--\eqref{posi_algum} or \eqref{posi_algum}--\eqref{dife_esti} hold; and additionally \eqref{lim_infinito}. Suppo\-se also that $a(x)$ and $f(x,t)$ satisfy either one of the following conditions
\begin{enumerate}[label=(\roman*)]
\item\label{teo_sub_UM} $b(x)\equiv 0,$ \eqref{V_pe}--\eqref{V_sirakov}, \eqref{increasing} and \eqref{minimax_comper}; or
\item\label{teo_sub_DOIS} $V(x)\geq 0,$ $b(x)$ has compact support, \eqref{V_sirakov}--\eqref{V_as}, \eqref{lim_infinito_infty} and \eqref{minimax_comper'}; or
\item\label{teo_sub_TRES} Replace condition \eqref{minimax_comper} $($respectively \eqref{minimax_comper'}$)$ in \ref{teo_sub_UM} $($respectively \ref{teo_sub_DOIS}$)$ by 
\begin{equation}\label{replace_improved}
I(u)  \leq I_\mathcal{P} (u) \quad (\text{respectively }I(u) \leq I_\infty (u) ),\quad\forall \, u \in H^s_{V}(\mathbb{R}^N).
\end{equation}
\end{enumerate}
Then Eq. \eqref{P} possess a nontrivial weak solution $u$ in $H^s_{V} (\mathbb{R}^N)$ at the mountain pass level, that is, $I(u)=c(I).$ Moreover, under the assumptions of items \ref{teo_sub_UM} and \ref{teo_sub_DOIS}, any sequence $(u_k)$ in $H^s_{V} (\mathbb{R}^N)$ such that $I (u_k) \rightarrow c(I)$ and $I' (u_k) \rightarrow 0$ has a convergent subsequence.
\end{theorem}
Theorems \ref{teoremao} \ref{GS_UM} and \ref{teo_sub} extend and complement some results of \cite{reinaldo,secchi,tintapaper} for the fractional framework. In Theorem \ref{teo_sub} the potential $a(x) = V(x) - b(x)$ is not necessarily bounded from below and in Theorem \ref{teo_sub} \ref{teo_sub_DOIS} we do not ask \eqref{increasing} as it was made in these works.
\begin{theorem}\label{teo_crit}
Assume that $f(x,t)$ and $a(x)\equiv V(x)$ satisfy \eqref{bem_def_ast}--\eqref{increasing_crit}, \eqref{suficient_ast}, \eqref{A-R}--\eqref{posi_algum} and \eqref{V_sirakov_ast}--\eqref{V_conve_seq_crit} respectively. Then Eq. \eqref{P} has a nontrivial weak solution in $\mathcal{D}^{s,2}(\mathbb{R}^N)$ at the mountain pass level. If we assume additionally condition \eqref{suficient_ast_0}, then any sequence $(u_k)$ in $\mathcal{D}^{s,2}(\mathbb{R}^N) $ such that $I_\ast (u_k) \rightarrow c(I_\ast)$ and $I_\ast' (u_k) \rightarrow 0$ has a convergent subsequence.
\end{theorem}
Theorems \ref{teoremao} \ref{GS_DOIS} and \ref{teo_crit} complement the study made in \cite{peral,tintahardy}. Theorem \ref{teo_crit} can be seen as a nonlocal version of \cite[Theorem 5.2]{tintahardy}, since we take into account that the critical non\-linearity is not autonomous. It also can be seen as a complement for many results in the literature about existence of nontrivial weak solution for Schr\"{o}dinger equation involving critical nonlinearities and singular potentials (cf. \cite{terracini, smets,felli_pistoia,felli_susanna} and the references given there), because we consider a general class that include as a particular case the inverse fractional square potential \eqref{par_hardy}.
\subsubsection{Remark on the hypothesis}
\begin{remark}\label{remarkao} Some comments on our assumptions are in order.
\begin{enumerate} [label=(\roman*)]
\item Assumption \eqref{bem_def} can be seen as a subcritical version of \eqref{selfsimilar} in the sense that it is oscillating about a subcritical power $|t|^{p-2}t,$ $2<p<2_s^\ast$ (see \cite{tintapaper} for the local case). We can see that \eqref{bem_def} holds if $f(x,t)$ satisfies the following conditions:
\begin{flalign}\tag{$f_1'$}\label{bem_def'}
	\hspace{1cm}
	\lim_{t\rightarrow 0}\frac{f(x,t)}{|t|+|t|^{2_s^\ast-1}}=0,\text{ uniform in }x.
	&&
\end{flalign}
\begin{flalign}\tag{$f_1''$}\label{bem_def''}
	\hspace{1cm}
\begin{aligned}
&\exists \ \varrho(t) \in C( \mathbb{R}\setminus \{0\})\cap L^\infty ( \mathbb{R}) \text{ s.t. }2<\inf_{t \in \mathbb{R}}\varrho (t)\leq  \sup_{t \in \mathbb{R}}\varrho(t)<2_s ^\ast\text{ and }\\
&|f(x,t)|\leq C(1+|t|^{\varrho(t)-1})\quad \text{a.e. } x \in \mathbb{R}^N,\ \forall \,  t \in \mathbb{R};
\end{aligned}
&&
\end{flalign}
Notice that $f(x,t)=k(x)\left[ \varrho'(t) (\ln |t| t) + \varrho (t)\right] |t|^{\varrho(t) -2}t,\ f(x,0)\equiv 0,$ satisfies \eqref{bem_def'}--\eqref{bem_def''}, for
\begin{equation*}
\varrho(t) = \frac{2_s ^\ast -2}{16}\sin\left( \ln( |\ln|t| |)\right)+\frac{5 2_s ^\ast+6 }{8}\quad\text{and}\quad0\leq  k(x)\in C(\mathbb{R})\cap L^\infty(\mathbb{R}^N).
\end{equation*}

\item Conditions \eqref{quadratic} and \eqref{dife_esti} imply that (see \cite[Lemma 2.1]{reinaldo}) there exists $p \in (2,2_s ^\ast)$ such that $\forall \ \varepsilon,\ \exists \ C_\varepsilon>0 \text{ s.t. }
|f(x,t)|\leq \varepsilon |t| + C_\varepsilon |t|^{p -1},\ \text{a.e. } x \in \mathbb{R}^N,\ \forall \,  t \in \mathbb{R}.$ Note that this is a special case of \eqref{bem_def}, precisely when $p_\varepsilon = p.$
\item Assumption \eqref{A-R} is the Ambrosetti-Rabinowitz condition which implies the mountain pass geometry and the boundedness of PS sequences for the associated functional (see for instance \cite{Rabi1992}). Conditions \eqref{quadratic}--\eqref{dife_esti} are an alternative for \eqref{A-R}, and was first introduced in \cite{dinglee} for the local case. By similar arguments in \cite{dinglee}, \eqref{dife_esti} holds if we assume \eqref{quadratic}, \eqref{dife_posi} and that there exist $p \in (2,2_s ^\ast)$ and $c_1,c_2,r_1>0$ such that
\begin{equation*}
|f(x,t)|\leq c_1 |t|^{p-1}\quad\text{and}\quad F(x,t)\leq \left(\frac{1}{2} - \frac{1}{c_2|t|^{\nu}} \right)f(x,t)t,\quad\forall \, |t| \geq r_1.
\end{equation*}
where $1<\nu <2$ if $N=1,$ and $1<\nu < N+p - p N/2s$ if $N\geq 2.$
\item In view of the boundedness of PS sequences, we separate our analysis for the subcritical case in two distinct situations: $f(x,t)$ satisfies \eqref{bem_def}--\eqref{posi_algum} or \eqref{posi_algum}--\eqref{dife_esti}. The first one is associated to the case where $f(x,t)$ has an oscillatory behavior around the subcritical power and the second one refers to the case where $f(x,t)$ does not satisfy the Ambrosetti-Rabinowitz condition.
\item In \cite{reinaldo}, considering the local case of Schr\"odinger equations with asymptotically periodic terms, it was proved the mountain pass geometry assuming $F(x,t) > 0$ for all $(x,t) \in \mathbb{R}^N \times \mathbb{R}$ and \eqref{quadratic} instead of the classical Ambrosetti-Rabinowitz condition. Here, in this work, we have an improvement even to the local case because we assume \eqref{posi_algum} instead of assuming that $F(x,t) > 0$ for all $(x,t) \in \mathbb{R}^N \times \mathbb{R}.$
\item Assumption \eqref{dife_posi} was used to prove the boundedness of PS sequences at the mountain pass level for the functional of Eq. \eqref{P}. In \cite{reinaldo} to prove similar result the author used the more restrictive condition
$\mathscr{F}(x,t) = \frac{1}{2}f(x,t)t - F(x,t)\geq b(t)t^2,$ for all $(x,t) \in \mathbb{R}^N \times \mathbb{R},$ for some $b(t) \in C(\mathbb{R}\setminus \{0\} , \mathbb{R}^+).$
\item To study the existence of weak solutions of Eq. \eqref{P}, we use \eqref{lim_infinito}, unlike the aforementioned papers, where the authors impose the more tight condition $|f(x,t) - f_\mathcal{P} (x,t)|\leq h(x)|t|^{q-1},$ a.e. $x$ in $\mathbb{R}^N$ and for all $t\text{ in }\mathbb{R},$ for some $h(x)\in C(\mathbb{R}^N)\cap L^\infty (\mathbb{R}^N)$ such that for any $\varepsilon > 0 ,$ the set $\{x \in \mathbb{R}^N : |h(x)| \geq \varepsilon\}$ has finite Lebesgue measure.
\item Condition \eqref{lim_infinito_infty} is used in the literature to prove that weak solutions of Eq. \eqref{P} satisfy a Pohozaev type identity.
\item We prove in Proposition \ref{welldef} that $H^s_V(\mathbb{R}^N)$ is well defined and it is continuously embedded in $H^s(\mathbb{R}^N),$ and consequently, the infimum $\mathcal{C}_V^{(\beta)} $ defined in \eqref{V_be} is strictly positive.
\item The class of functions satisfying \eqref{selfsimilar} can be seen as nonlinearities asymptotically oscillating about the critical power $|t|^{2_s^\ast-2}t$ and was introduced in \cite{paper1,tinta_pos}.
\item The asymptotic additivity in \eqref{f_extra} ensures the convergence of $I_\ast$ under the weak profile decomposition for bounded sequences in $\mathcal{D}^{s,2}(\mathbb{R}^N)$ described in Theorem \ref{teo_tinta_frac} (see also \cite{paper1}).
\item As already mentioned, \eqref{V_sirakov_ast}--\eqref{V_conve_seq_crit} define a class of singular potentials that vanishes at infinite, see Example \ref{r_ex}--(iv).
\item Once the limits in \eqref{V_as}, \eqref{lim_infinito}, \eqref{lim_infinito_infty} or \eqref{limites_ast} exist, to obtain compactness of PS sequences at the minimax levels we need to require the additional conditions over the minimax levels $c_\mathcal{P},c_\infty, c_0, c_+, c_{-}$ given in \eqref{minimax_comper}, \eqref{minimax_comper'}, \eqref{suficient_ast}--\eqref{suficient_ast_0}. In fact, we do not believe that it is possible, in general, to achieve the compactness described in Theorems \ref{teo_sub} and \ref{teo_crit} without these conditions. This approach was introduced by P.-L. Lions in \cite{lionscompcase1,lionscompcase2,lionslimitcase1,lionslimitcase2}.
\item We also consider the case when \eqref{minimax_comper}, \eqref{minimax_comper'}, \eqref{suficient_ast}--\eqref{suficient_ast_0} do not hold. Precisely, when $c(I) = c(I_\mathcal{P})$ or $c(I) = c(I_\infty).$ In this case, we can not use the concentration-compactness argument at the mountain pass level. We apply \cite[Theorem~2.3]{lins} to overcome this difficulty and prove existence of solution at the mountain pass level.
\item For problem \eqref{P} involving critical growth we require \eqref{V_sirakov_ast}--\eqref{V_conve_seq_crit} and \eqref{limites_ast}--\eqref{suficient_ast_0}. These assumptions are suitable for our argument, differently from \eqref{minimax_comper}--\eqref{minimax_comper'}, because the potential that appears in the associated limiting equation depends on the profile decomposition of Theorem \ref{teo_tinta_frac} for a given PS sequence at the mountain pass level (for more details see estimate \eqref{reasonbefore}).
\end{enumerate}
\end{remark}

\begin{remark}
	Under \eqref{V_as} and \eqref{lim_infinito} the next conditions imply that \eqref{minimax_comper} and \eqref{minimax_comper'} hold:
\begin{flalign}\label{suf_cond_sub}\tag{$\mathscr{H}$}
\begin{aligned}
&F_\mathcal{P}(x,t) \leq  F(x,t),\quad \text{a.e. } x \in \mathbb{R}^N ,\ \forall \, t \in \mathbb{R},\text{ and } V(x)\leq V_\infty,\quad\text{a.e. }x\in\mathbb{R}^N.\\
&\text{Moreover, the first inequality holds strictly in some open interval contained the origin}\\
&\text{or the second one holds in a set of positive measure.}
\end{aligned}
&&
\end{flalign}
In Proposition \ref{p_assumption_faz_sentido} we proved the following estimates for the minimax levels, $c(I) \leq c(I_\mathcal{P})$ and $c(I) \leq c(I_\infty).$ Moreover, we proved that under \eqref{suf_cond_sub}, \eqref{minimax_comper} and \eqref{minimax_comper'} hold. We observe that on the corresponding assumption of Theorem \ref{teo_sub}, it is easy to see that \eqref{suf_cond_sub} imply that \eqref{replace_improved} is satisfied.
\end{remark}

\begin{example}\label{r_ex}
Our approach include the following classes of potentials:
\begin{enumerate}[label=(\roman*)]
\item For a potential satisfying assumption \eqref{V_sirakov} and that is not bounded away from zero, consider $0\leq a(x)\equiv V_0(x) \in L^p _{\loca} (\mathbb{R}^N) \cap (C(\mathbb{R}^N\setminus \mathcal{O}),$ where $p\geq 1$ and $\mathcal{O}$ is a countable set, and suppose that $Z = \{x \in \mathbb{R}^N : V(x)=0 \}\neq \emptyset$ is a countable discrete set.
\item Let $V_0(x)$ the potential given above. For a potential the changes sign and satisfies \eqref{V_sirakov}, consider $a(x)\equiv  V_0(x) - \varepsilon,$ where $0<\varepsilon <\mathcal{C}_{V_0}/2.$
\item To study potential of the form $a(x) = V(x) - b(x),$ taking $ 0<\delta <N / \beta,$ $p>N/s$ and 
\begin{equation*}
V(x) =2-\frac{1}{1+|x|^2},\quad V_\infty = 2\quad \text{and}\quad b(x)=
\left\{
\begin{aligned}
\mathcal{C}_b |x|^{-\delta},&\quad\text{if }|x|\leq 1,\\
0,&\quad\text{if }|x|>1.
\end{aligned}
\right.
\end{equation*}
$a(x)=V(x)-b(x)$ satisfies \eqref{V_sirakov}--\eqref{V_as}, $\mathcal{C}_b>0$ is a normalization constant.
\item For potential $a(x) \equiv V(x)$ satisfying \eqref{V_sirakov_ast}--\eqref{V_conve_seq_crit}, in view of \eqref{par_hardy}, we can consider
\begin{equation*}
V(x) = - \frac{1}{L}\sum _{j = 1} ^L \frac{\lambda_j}{|x-a_\ast|^{2s}},\quad\text{with}\quad 0<\lambda _j < \frac{\Gamma _{N,s}}{2},\quad j=1,\ldots,L,
\end{equation*}
\end{enumerate}
\end{example}
\begin{example}
Hypotheses of Theorem \ref{teoremao}--\ref{teo_crit} are satisfied by:
\begin{enumerate}[label=(\roman*)]
\item Taking $\varrho(t)$ as in Remark \ref{remarkao}--(i) and $k(x) = |x|^2/(1+|x|^2),$ one can see that $f(x,t)=k(x)\left[ \varrho'(t) (\ln |t| t) + \varrho (t)\right] |t|^{\varrho(t) -2}t,$ $f(x,0)\equiv 0,$ satisfies \eqref{bem_def}--\eqref{posi_algum}, \eqref{lim_infinito_infty} and \eqref{minimax_comper'}.
\item For a nonlinearity satisfying conditions \eqref{posi_algum}--\eqref{dife_esti}, \eqref{lim_infinito}, \eqref{increasing} and \eqref{minimax_comper} we can define $f(x,t) = h(x,t)$ for $t\geq 0,$ and $f(x,t) = -h(x,-t),$ for $t<0,$ where\\$h(x,t)=k(x)t \ln (1+t) + k_1(x) \left[(1+ \cos (t))t^2 + 2 (t + \sin (t))t \right],$ for $ t\geq 0,\ s>N/6;$ $k(x) = |x|^2/(1+|x|^2)$ and $0\leq k_1 (x) \in C(\mathbb{R}^N)$ satisfies $\lim _{|x| \rightarrow \infty} k_1(x) = 0.$

\item Let $c(x)$ be a continuous nonnegative $\mathbb{Z}^N$--periodic function and $f(x,t)=c(x)\left[p h_\varepsilon(t)+h'_\varepsilon(t)t\right]|t|^{p-1},$ $2<p<2_s ^\ast,$ with $h_\varepsilon(t)\in C^\infty(\mathbb{R})$ is a nondecreasing cutoff function satisfying $|h'_\varepsilon(t)|\leq C/t,\ |h_\varepsilon(t)|\leq C,$ for all $t \in \mathbb{R},$ $h_\varepsilon(t)=-\varepsilon,$ for $t\leq 1/4,$ $h_\varepsilon(t)=\varepsilon,$ for $t\geq 1/4,$ with $\varepsilon$ small enough. In this case $F(x,t)$ may change sign.
\item The nonlinearity $f(x,t) = \exp\{ k_0(x)( \sin (\ln |t|) +2 )  \} \left[ k_0(x) \cos (\ln |t|) + 2^{\ast} _s\right]  |t| ^{2^{\ast} _s -2}t,$ $f(x,0)\equiv0,$ satisfies the hypothesis of Theorem \ref{teo_crit}, if $k_0(x)$ is continuous and $2_ s ^\ast - \mu >\sup _{x \in \mathbb{R}^N} k_0(x)\geq k_0(x)>k_0(0) = \inf _{x \in \mathbb{R}^N} k_0(x) = \lim _{|x| \rightarrow \infty}k_0(x)=0.$
\end{enumerate}
\end{example}
\section{Preliminaries}\label{s_Preliminaries}
\subsection{Fractional Sobolev spaces} Let $0<s<N/2,$ by Plancherel Theorem, we have
\begin{equation}\label{planche_equiv}
[u]_s^2 = \int _{\mathbb{R}^N} |(-\Delta)^{s/2}u|^2 \dx,\quad\forall \, u \in C^\infty _0 (\mathbb{R}^N).
\end{equation}
Thus $\mathcal{D}^{s,2} (\mathbb{R}^N)$ is a Hilbert space when endowed with the inner product $[u,v]_s = \int _{\mathbb{R}^N} (-\Delta )^{s/2} u (-\Delta )^{s/2 } v \dx,$ and the following characterization holds $\mathcal{D}^{s,2}(\mathbb{R}^N) = \left\{ u \in L ^{2^{\ast } _s } (\mathbb{R}^N ) : \ (-\Delta )^{s/2} u \in  L^2 (\mathbb{R} ^N) \right\}.$

By \eqref{planche_equiv} we also have that $H^s (\mathbb{R}^N)$ is a Hilbert space with norm and inner product
\begin{equation*}
\| u\| ^2 :=\int _{\mathbb{R}^N} |(-\Delta)^{s/2} u| ^2 \dx +u^2\dx,\quad (u,v) := \int _{\mathbb{R}^N} (-\Delta )^{s/2} u (-\Delta )^{s/2} v +u v \dx.
\end{equation*}
Thus $H ^s (\mathbb{R} ^N) =\left\lbrace u \in L^2 (\mathbb{R}^N) : (-\Delta )^{s/2} u \in  L^2 (\mathbb{R} ^N) \right\rbrace.$ For $\Omega \subset \mathbb{R}^N$ a $C^{0,1}$ domain with bounded boundary and  $0<s<1,$ the fractional Sobolev space is defined as
\begin{equation*}
H ^s (\Omega)=\left\lbrace u \in L^2 (\Omega) :  \int _{\Omega} \int _{\Omega} \frac{\left| u(x) - u(y) \right|^2}{|x-y|^{N + 2s}}dxdy <\infty \right\rbrace,
\end{equation*}
with the norm $\| u \| _{H ^s (\Omega )} ^2 := \int _{\Omega} u^2\dx + \int _{\Omega} \int _{\Omega} \frac{\left| u(x) - u(y) \right|^2}{|x-y|^{N + 2s}}\dxdy .$ By \cite[Proposition 3.4]{hitchhiker}, we have
\begin{equation*}
[u]_s^2 = \frac{C(N,s)}{2} \int _{\mathbb{R}^N} \int _{\mathbb{R}^N} \frac{ |u(x) - u(y) |^2}{|x-y|^{N + 2s}}\dxdy,\quad\forall \, u \in \mathcal{D}^{s,2} (\mathbb{R}^N).
\end{equation*}
Moreover, we have the continuous embedding
\begin{equation}\label{Hembd}
H^s (\Omega) \hookrightarrow L^p(\Omega),\quad 2 \leq p \leq 2_s ^\ast,\quad\text{for}\quad 0 < s < N/2,
\end{equation}
and the following compact embedding (see \cite[Section 7]{hitchhiker}),
\begin{equation}\label{comp_loc}
\mathcal{D}^{s,2}(\mathbb{R}^N)\hookrightarrow L _{\loca}^p (\mathbb{R}^N),\quad 1 \leq p < 2_s ^\ast,\quad\text{for}\quad 0 < s < \min\{1,N/2\}.
\end{equation}
Thus, any bounded sequence in $H^s(\mathbb{R}^N)$ has subsequence that converges strongly in $L^p(\Omega),$ $1 \leq p < 2_s ^\ast,$ for any compact set $\Omega$ of $\mathbb{R}^N.$ The Plancherel Theorem also gives the next identity,
\begin{equation}\label{formula}
\int _{\mathbb{R}^N} (-\Delta )^{s/2} u(-\Delta )^{s/2 } v \dx = \int _{\mathbb{R}^N} (-\Delta )^{s} u  v \dx,\quad\forall \,u\in H^{2s}(\mathbb{R}^N),\ v \in H^{s}(\mathbb{R}^N).
\end{equation}
\subsection{The $s$-harmonic extension}
Next we introduce the harmonic extension following \cite[Section 2]{frac_niremberg}. Let $P _s(x,y) = \beta(N,s)\frac{y^{2s}}{\left(|x|^2 + y^2 \right)^{\frac{N+2s}{2}} },$ where $\beta(N,s)$ is such that $\int _{\mathbb{R^N}} P _s (x,1) \dx =1$ and $0<s<1.$ Considering the standard notation $\mathbb{R}^{N+1} _+=\{(x,y)\in \mathbb{R}^{N+1}:y>0 \},$  for $u \in \mathcal{D}^{s,2} (\mathbb{R}^N)$ let us set the $s-$harmonic extension of $u$,
\begin{equation*}
w(x,y)= E_s (u) (x,y) := \int _{\mathbb{R}^N} P _s (x- \xi,y) u (\xi) \dxi,\quad (x,y) \in \mathbb{R}_+ ^{N+1}.
\end{equation*}
Then, for $K\subset\overline{\mathbb{R}^{N+1} _+}$ compact we have $w \in L^2 (K, y^{1-2s})$, $\nabla w \in L ^2 (\mathbb{R}^{N+1} _+, y^{1-2s})$ and  $w \in C^{\infty} (\mathbb{R}^{N+1} _+).$ Moreover, $w$ satisfies in the distribution sense
\begin{equation}\label{sol_div}
\left\{
\begin{aligned}
\dive (y^{1-2s} \nabla w ) & = 0,  & \text{in } & \mathbb{R}^{N+1} _+,  \\
-\lim _{y \rightarrow 0 ^+} y^{1-2s} w_y (x,y) & = \kappa _s (- \Delta) ^s u (x)   & \text{in } & \mathbb{R}^{N},\\
\| \nabla w \| ^2 _{ L ^2 (\mathbb{R}^{N+1} _+, y^{1-2s}) } &=\kappa _s \|u\|^2,
\end{aligned}
\right.
\end{equation}
where $\kappa _s = 2^{1-2s}  \Gamma (1-s) / \Gamma (s),$ and $\Gamma$ is the gamma function. Precisely, for $R>0$,
\begin{equation}\label{ext_distrib_sense}
\int _{B_R ^+} y^{1-2s} \left\langle  \nabla w , \nabla \varphi \right\rangle \dxdy = \kappa _s \int_{B_R ^N} (-\Delta) ^{s/2} u (-\Delta) ^{s/2} \varphi \dx, \quad \forall \,  \varphi \in C ^\infty _0 (B ^+_R \cup B^N _R),
\end{equation}
where $B_R := \{ z=(x,y) \in \mathbb{R}^{N+1}  :|z|^2<R^2 \}, B_R ^+ := B_R \cap \mathbb{R}^{N+1} _+,$ $B ^N _R := \{ z=(x,y) \in \mathbb{R}^{N+1} _+  :|z|^2<R^2, \  y = 0 \}$ and the right-hand side of \eqref{ext_distrib_sense} is in the trace sense for $\varphi$ (for details on the trace operator see \cite{nekvinda}). More generally, given $g:\mathbb{R}^N \times \mathbb{R} \rightarrow \mathbb{R},$ $v \in H^1 (B ^+ _R , y ^{1-2s} )$ is a weak solution of the problem
\begin{equation}\label{weak_final}
\left\{
\begin{aligned}
\dive (y^{1-2s} \nabla v )  = &  0  \quad  & \text{in} \quad  B^+_R,\\
-\lim _{y \rightarrow 0 ^+} y^{1-2s} v_y (x,y)  = & \kappa _s g(x,v(x) ) \quad & \text{in} \quad B^N _R,
\end{aligned}
\right.
\end{equation}
if we have
\begin{equation}\label{defi_fraca}
\int _{B_R ^+} y^{1-2s} \left\langle  \nabla v , \nabla \varphi \right\rangle \dxdy = \kappa _s \int_{B_R ^N} g(x,v )  \varphi  \dx,\ \forall \, \varphi \in C ^\infty _0 (B ^+_R \cup B^N _R).
\end{equation}
Let $g(x,t)=f(x,t)-a(x)t$ and $u \in \mathcal{D}^{s,2}(\mathbb{R}^N)$ be such that $f(u),\ F(u)\in L^1(\mathbb{R}^N).$ Let $V(x)\in L^1_{\loca} (\mathbb{R}^N),$ satisfying \eqref{V_sirakov} and $b(x)$ verifying \eqref{V_be}. Then $w=E_s (u)$ is a weak solution of \eqref{weak_final} for all $R$ if, and only if, $u$ is a weak solution of \eqref{P}.
\begin{remark}\label{r_nonnegative}
Using the $s$-harmonic extension, it can be proved the existence of nonnegative weak solutions of \eqref{P} if $f(x,t)\geq 0$ for all $t\geq 0$ and a.e. $x$ in $\mathbb{R}^N.$ For that one can consider the truncation $\bar{f}(x,t)= f(x,t),$ if $t \geq 0,$ $\bar{f}(x,t)= 0,$ if $t<0.$ Assume that $a(x) \in L^1 _{\loca}(\mathbb{R}^N)$ and that \eqref{bem_def} and \eqref{V_sirakov} hold true with $b(x)\equiv 0.$ Thus for $u$ a weak solution of \eqref{P}, with $f(x,t)$ replaced by $\bar{f}(x,t),$ we have that $u$ is also a weak nonnegative solution for \eqref{P}. To see that, let $\xi \in C_0 ^\infty (\mathbb{R}:[0,1])$ such that $\xi(t)=1,$ if $t \in [-1,1]$ and $\xi(t)=0,$ if $|t|\geq 2,$ with $|\xi ' (t)| \leq C,$ $\forall t\in \mathbb{R}.$ For each $n\in \mathbb{N},$ define $\xi _n : \mathbb{R}^{N+1} \rightarrow \mathbb{R}$ by $\xi _n (z) = \xi (|z|^2/n^2).$ Then $\xi _n \in C_0 ^\infty (\mathbb{R}^{N+1})$ and verifies $|\nabla \xi _n (z)| \leq C$ and $|z||\nabla \xi _n (z)| \leq C,\ \forall \,z\in \mathbb{R}^{N+1}.$ By a density argument, we can take $\varphi = \xi _n w_{-}$ in \eqref{defi_fraca}, where $w_{-} (z) = \min\{w(z),0\}.$
Since $w_{-} (z)=E_s(u_{-}),$ we have
\begin{multline*}
\int_{\mathbb{R}_+^{N+1}}y^{1-2s}\xi_n |\nabla w_{-}|^2 + y^{1-2s}\xi_n \left\langle \nabla w _{+}, \nabla w_{-} \right\rangle +y^{1-2s} \left\langle \nabla w_{+}+\nabla w_{-},w_{-} \nabla \xi_n \right\rangle \dxdy \\ = \kappa _s \int_{\mathbb{R}^N} (\bar{f}(x,u)-a(x)u )\xi _n u_{-}\dx,
\end{multline*}
applying the Lebesgue Theorem and \eqref{sol_div} we get $\| u _{-}\|^{2} _{V} = \int_{\mathbb{R}^N} \bar{f}(x,u)u_{-} \dx=0,$ thus $u_{-}=0.$ If $u$ has sufficient regularity one can show that $u$ is positive, by applying the maximum principle as described in \cite{silvestre}. In order to regularize solutions of Eq. \eqref{P}, we can follow \cite[Section 6]{secchi}.
\end{remark}
\section{Proof of Theorem \ref{teo_tinta_frac_sub}}\label{s_proof_profile_decomp}
We shall prove the profile decomposition for bounded sequences in $H ^s (\mathbb{R} ^N),$ $0<s \leq N/2.$ To achieve that we start by considering
\begin{equation*}
D=D_{\mathbb{Z}^N } := \left\lbrace g_y : H ^s (\mathbb{R} ^N) \rightarrow H ^s (\mathbb{R} ^N) : g_y u(x) = u(x-y), \ y \in \mathbb{Z}^N \right\rbrace,
\end{equation*}
which turns to be a group of unitary operators in $H ^s (\mathbb{R} ^N).$ The idea is to use \cite[Theorem~3.1]{tintabook} to obtain Theorem \ref{teo_tinta_frac_sub}. We need first to determine how elements of $H^s(\mathbb{R}^N)$ becomes asymptotically orthogonal in $H^s(\mathbb{R}^N)$ with respect to any fixed function under a sequence of dislocations.
\begin{lemma}\label{eleminha}
Let $(y_k)$ be a sequence in $\mathbb{R} ^N$ and $u \in H^s(\mathbb{R}^N)\setminus\{0\}.$ The sequence $( u (\cdot - y_k ) )$ converges weakly to zero in $H^s(\mathbb{R}^N)$ if, and only if $|y_k| \rightarrow \infty.$
\end{lemma}
\begin{proof}
Suppose that $u (\cdot - y_k ) \rightharpoonup 0$ in $H^s(\mathbb{R}^N),$ and assume by contradiction, that $y_k \rightarrow y$ up to a subsequence. By density we may assume that $u \in C^\infty_0(\mathbb{R}^N),$ also by \cite[Lemma 5.1]{paper1} we have that $u (\cdot - y_k ) \rightarrow u (\cdot - y)$ in $\mathcal{D}^{s,2}(\mathbb{R}^N),$ consequently by the Dominated Convergence Theorem
\begin{multline*}
0=\lim_{k \rightarrow \infty}(u(\cdot - y_k),u(\cdot - y)) \\= \lim_{k \rightarrow \infty}\left[  
\int _{\mathbb{R}^N} (-\Delta )^{s/2} u(\cdot - y_k) (-\Delta )^{s/2 } u(\cdot - y) + u(\cdot - y_k)u (\cdot - y)\dx 
\right]  = \|u\|^2,
\end{multline*}
which leads to a contradiction with the assumption that $u \neq 0.$ Conversely, assume that $|y_k| \rightarrow \infty.$ Again, by density argument we may assume $u \in C^\infty _0 (\mathbb{R}^N),$ and using \cite[Lemma 5.2]{paper1} we obtain that $u (\cdot - y_k ) \rightharpoonup 0$ in $\mathcal{D}^{s,2}(\mathbb{R}^N).$ Since $\supp (u (\cdot - y_k)) \cap \supp (v) = \emptyset, $ for $k $ large enough, we have
\begin{equation*}
\lim_{k \rightarrow \infty}\left[  
\int _{\mathbb{R}^N} (-\Delta )^{s/2} u(\cdot - y_k) (-\Delta )^{s/2 } v + u(\cdot - y_k)v \dx 
\right]  =0,\quad \forall \,v\in C^\infty _0 (\mathbb{R}^N),\qedhere
\end{equation*}
\end{proof}
We complement \cite{tintacocompact} establishing an equivalence between the $L^p$-convergence and $D_{\mathbb{Z}^N }$-convergence (see \cite[Definition 3.1]{tintabook} or \cite[Definition 1.1]{tintacocompact}) in $H^s (\mathbb{R}^N).$ Thus, Theorem \ref{teo_tinta_frac_sub} follows by an argument in \cite[Corollary 3.3]{tintabook}.
\begin{proposition}\label{cocompact}
	Let $(u_k)$ be a bounded sequence in $H^s (\mathbb{R}^N).$ Then $u_k \stackrel{D_{\mathbb{Z}^N }}{\rightharpoonup} 0$ in $H^s (\mathbb{R}^N),$ if and only if $u_k \rightarrow 0$ in $L^{p}(\mathbb{R}^N),$ for all $2<p<2^\ast _s.$
\end{proposition}
\begin{proof}
Suppose that $u_k \rightarrow 0$ in $L^{p}(\mathbb{R}^N),$ $2<p<2^\ast _s.$ Take an arbitrary sequence $(g_{y_k}) $ in $D_{\mathbb{Z}^N}$ and let $\varphi \in C ^\infty _0 (\mathbb{R}^N).$ Using \eqref{formula} we have
\begin{equation*}
\left|\int_{\mathbb{R}^N} (-\Delta )^{s/2} (g^\ast_{y_k}u_k) (-\Delta )^{s/2} \varphi \dx \right| \leq \left(\int_{\mathbb{R}^N} |u_k| ^p \dx\right)^{\frac{1}{p}} \left(\int_{\mathbb{R}^N}|(-\Delta ) ^s \varphi(\cdot-y_k)|^{\frac{p}{p-1}} \dx \right) ^{\frac{p-1}{p}}.
\end{equation*}
Thus, using H\"{o}lder inequality in the $L^2$ term of the inner product of $H ^s (\mathbb{R}^N)$, we conclude that $g^\ast_{y_k} u_k \rightharpoonup 0$ in $H ^s (\mathbb{R}^N).$ For the rest of the proof we refer the reader to \cite[Theorem 2.4]{tintacocompact}.
\end{proof}
\begin{proof}[Proof of Theorem \ref{teo_tinta_frac_sub} completed]We prove by applying \cite[Theorem~3.1]{tintabook}. In fact, let $(g_{y_k})$ in $D_{\mathbb{Z}^N }$ such that $g_{y_k} \not \rightharpoonup 0$ in $H^s(\mathbb{R}^N).$ By Lemma \ref{eleminha}, up to a subsequence $y_k \rightarrow y,$ and by \cite[Lemma 5.2]{paper1} $g_{y_k}\rightarrow g_y.$ Thus, in view of \cite[Proposition~3.1]{tintabook}, $(H^s(\mathbb{R}^N),D_{\mathbb{Z} ^N})$ is a dislocation space. Assertions \eqref{seis.dois.sub} and \eqref{seis.quatro.sub} follows by Lemma \ref{eleminha} and Proposition \ref{cocompact} respectively.
\end{proof}
\section{Variational settings}\label{s_variational}
In this section we have the basic background to apply variational argument to study \eqref{P}.
\begin{proposition}\label{welldef} Let $V(x)\in L ^1 _{\loca} (\mathbb{R}^N)$ satisfying \eqref{V_sirakov}, then $H^s _{V}(\mathbb{R} ^N)$ is a Hilbert space continuously embedded in $H^s(\mathbb{R}^N).$ If $V(x)$ satisfies \eqref{V_sirakov_ast}, then $\| \cdot \|_{V}$ is equivalent to the  norm of $\mathcal{D}^{s,2}(\mathbb{R}^N).$
\end{proposition}
\begin{proof}
Let us prove first that there exists a positive constant $C$ such that
\begin{equation}\label{prova_imersao}
C [ \varphi ]_s ^2 \leq \|\varphi \|^2 _{V},\quad\forall \, \varphi \in C^\infty _0 (\mathbb{R}^N).
\end{equation}
In fact, on the contrary, there would exist a sequence $(\varphi_n)$ in $C_0^\infty (\mathbb{R}^N ),$ such that 
\begin{equation*}
[\varphi _n]_s ^2 > n \| \varphi _n \|_{V}^2,\quad\forall \, n\in \mathbb{N}.
\end{equation*}
Taking $v_n = \varphi _n/ [\varphi _n]_s,$ we have $\frac{1}{n}> \| v_n \|^2 _{V}$ and $\mathcal{C}_{V} \|v_n\|_2 ^2 \leq \| v_n \| _{V} ^2,$ $\forall \, n \in \mathbb{N},$ and consequently $\lim _{n \rightarrow \infty } \| v_n \| _{V} ^2 = \lim _{n \rightarrow \infty } \|v_n \|_2 ^2 = 0.$ This leads to a contradiction with the fact that $1-\mathcal{B} \|v_n \| _2 ^2 \leq \| v_n \| _{V} ^2,$ $\forall \, n \in \mathbb{N}.$ Now consider $(\varphi _n)$ any sequence in $C_0^\infty (\mathbb{R}^N ).$ Using inequality \eqref{prova_imersao} we have $C [ \varphi _m -\varphi_n]_s ^2 \leq \|\varphi _m- \varphi_n\|^2 _{V},\quad \forall \, m \neq n.$

Consequently,  $ \| \varphi _m- \varphi_n \| ^2 \leq \min \{ 1,C\} ^{-1} \left(1 + \mathcal{C}_{V}^{-1}\right) \| \varphi _m -\varphi_n \| _{V} ^2,$ $\forall \, m \neq n.$ Thus $H^s _{V}(\mathbb{R} ^N)$ is well defined. Moreover, Fatou Lemma and embedding \eqref{Hembd} imply $H^s _{V}(\mathbb{R} ^N) \subset \left\lbrace u \in H^s (\mathbb{R} ^N) : \int _{\mathbb{R}^N} V(x) u^2 \dx < \infty \right\rbrace,$ with the continuous embedding $H^s _{V}(\mathbb{R} ^N) \hookrightarrow H^s (\mathbb{R} ^N).$ Assuming \eqref{V_sirakov_ast},$
[u]_s ^2 + \int _{\mathbb{R}^N}V(x) u^2\dx \geq \mathcal{C}^\ast_{V} \int _{\mathbb{R}^N} |V(x)|u^2 \dx,$ $\forall \, u\in C^\infty _0 (\mathbb{R}^N),$ from this we derive
\begin{equation*}
\mathcal{C}^\ast_{V} [u]^2_s \leq (\mathcal{C}^\ast_{V}+1)[u]^2_s + \int_{\mathbb{R}^N}(V(x)-\mathcal{C}^\ast_{V} |V(x)|) u^2 \dx \leq (\mathcal{C}^\ast_{V} +1)\|u\|^2_{V},\quad\forall \, u\in C^\infty _0 (\mathbb{R}^N).
\end{equation*}
Since $V(x)\leq 0$ a.e. in $\mathbb{R}^N,$ the norms $[ \ \cdot \ ]_s$ and $\|\cdot\|_{V}$ are equivalent in $\mathcal{D}^{s,2}(\mathbb{R}^N).$
\end{proof}
\begin{remark}\label{assymp_defined}
\begin{enumerate}[label=(\roman*)]
\item If $V(x)$ fulfills \eqref{V_sirakov} and \eqref{V_as}, then $H_{V}^s (\mathbb{R} ^N) = H^s (\mathbb{R} ^N).$ Moreover, the norms $\|\cdot\|$ and $\|\cdot\| _{V}$ are equivalent. Consequently, the path $\lambda _u (t):= u(\cdot /t),$ $t \geq 0$ belongs to $C([0, \infty ),  H^s _{V}(\mathbb{R} ^N) )$ and $u (\cdot - y) \in H^s _{V} (\mathbb{R}^N)$ for all $u \in H^s _{V} (\mathbb{R}^N)$ and $y\in\mathbb{R}^N.$ Indeed, there is a ball $B_{R_1}$ with center at the origin such that
\begin{align*}
\qquad \quad \int _{\mathbb{R}^N} V(x) u^2 \dx &=\int _{B_{R_1}} V(x) u^2 \dx+ \int _{\mathbb{R}^N \setminus B_{R_1}} V(x) u^2 \dx\\
&\leq\left(\int _{B_{R_1}} |V(x)| ^{\sigma }\dx \right)^{1/\sigma}\left(\int _{B_{R_1}} |u| ^{2 \sigma/(\sigma - 1)} \dx\right)^{(\sigma -1)/ \sigma} \\&\qquad+(V_\infty + 1)\int _{\mathbb{R}^N \setminus B_{R_1}}u^2 \dx,\ \forall \, u \in H^s _{V} (\mathbb{R}^N ),&
\end{align*}
where $2\leq 2 \sigma/(\sigma - 1) \leq 2_s ^\ast.$ So we can apply \eqref{Hembd} to conclude. To obtain that $\lambda_u$ belongs to $H_{V}^s (\mathbb{R} ^N)$ we use \cite[Lemma 8.3]{paper1}.
\item If \eqref{V_pe}--\eqref{V_sirakov} hold, then Theorem \ref{teo_tinta_frac_sub} holds replacing $H^s(\mathbb{R}^N)$ by $H^s _{V} (\mathbb{R}^N)$ and $\| \cdot \|$ by $\| \cdot \| _V.$ In fact, \eqref{V_pe} implies that $D _{\mathbb{Z} ^N}$ is a group of unitary operators in $H^s_{V} (\mathbb{R}^N).$
\end{enumerate}
\end{remark}
\begin{lemma}\label{geometry}
Suppose that $f(x,t)$ satisfies \eqref{bem_def} and either \eqref{A-R}--\eqref{posi_algum} or \eqref{quadratic}. If $a(x)=V(x) - b(x)\in L^1_{\loca} (\mathbb{R}^N)$ fulfill \eqref{V_sirakov}--\eqref{V_be}, then $I$ possess the mountain pass geometry. Precisely, 
\begin{enumerate}[label=(\roman*)]
\item $I(0)=0;$
\item There exist $r,\ b>0$ such that $I(u) \geq b,$ whenever $\|u\| _{V} = r; $
\item There is $e\in H ^s _{V} (\mathbb{R}^N)$ with $\|e\|_{V} > r$ and $I(e)<0;$ 
\end{enumerate}
In particular $0<c(I)<\infty.$ 
\end{lemma}
\begin{proof}
Let $\xi _R \in C^\infty _0 (\mathbb{R}),\ R>0,$ such that $0\leq \xi _R(t)\leq t_0,$ $\xi _R(t)= t_0$ if $|t| \leq R,$ and $\xi _R(t)= 0$ if $|t| > R+1.$ Setting $v(x) := \xi_R(|x-x_0|),$ we have $v \in H^s _{V} (\mathbb{R}^N )$ and by assumption \eqref{posi_algum} we have
\begin{align*}
\int_{\mathbb{R}^N} F(x,v ) \dx &= \int _{B_R (x_0)} F(x,t_0) \dx + \int_{B_{R+1}(x_0)\setminus B_R(x_0)} F(x,v ) \dx\\
&\geq |B_R|\inf _{B_R(x_0)} F(x,t_0) + |B_{R+1}\setminus B_R|\inf_{(x,t) \in  C_R(x_0,t_0)} F(x,t) >0.
\end{align*}
First assume that \eqref{A-R} holds. Since $b(x) \in L^{\beta}(\mathbb{R}^N),$
\begin{equation*}
\int_{\mathbb{R}^N} b(x) u^2 \dx \leq \left(\int _{\mathbb{R}^N} |b(x)|^{\beta}\dx \right)^{1/ \beta} \left( \int _{\mathbb{R}^N}|u|^{2 \beta / (\beta -1)} \dx \right)^{(\beta -1) / \beta},\quad\forall \,  u \in H^s_{V}(\mathbb{R}^N),
\end{equation*}
with $2<2 \beta / (\beta -1) < 2_s ^\ast,$ by \eqref{bem_def} and \eqref{V_be}, for any $\varepsilon$ we get
\begin{equation}\label{eq_geo_posi}
I(u) \geq \left[\frac{1}{2} \left(1 -\frac{\|b(x)\|_{\beta}}{\mathcal{C}^{(\beta)}_V}-2\varepsilon \mathcal{C}_2\right)-\varepsilon\mathcal{C}_{2_s^\ast}\|u\|_{V}^{2_s ^\ast-2} - C_\varepsilon \mathcal{C}_{p_\varepsilon }\|u\|_{V}^{p _\varepsilon -2} \right] \|u\|^2_{V},\ \forall \, u \in H^s_{V}(\mathbb{R}^N),
\end{equation}
where $\mathcal{C}_2,$ $\mathcal{C}_{2_s^\ast}$ and $\mathcal{C}_{p_\varepsilon }$ are positive constants given in Proposition \ref{welldef}. This allows to consider $\varepsilon$ such that the first term in the right-hand side of \eqref{eq_geo_posi} is positive, once $\|u\|_{V}$ is taken small enough. Hence there exists $r>0$ such that $I(u) >0$ provided that $\|u\|_{V}=r.$ Since \eqref{A-R} is equivalent to $d/dt (F(x,t)t ^{-\mu}) \geq 0,$ for $t> 0,$ we have 
\begin{equation*}
\int_{\mathbb{R}^N}F(x,tv)\dx \geq t^\mu \int_{\mathbb{R}^N} F(x,v)\dx,\quad\text{whenever } t>1.
\end{equation*}
Hence, as $t \rightarrow \infty,$
\begin{equation*}
I(tv)=\frac{t^2}{2}\|v\|^2 _{V} -\int_{\mathbb{R}^N}b(x) u^2\dx - \int_{\mathbb{R}^N} F(x,tv)\dx \leq \frac{t^2}{2}\|v\|_V^2 - t^\mu \int_{\mathbb{R}^N} F(x,v)\dx \rightarrow - \infty,
\end{equation*}
Now suppose that \eqref{quadratic} holds. By Remark \ref{remarkao}--(ii) we can argue as above to conclude the existence of $r>0$ such that $I(u) > 0$ wherever $\|u\|_{V} <r.$ Given $R>0,$ there exists $t_{R}>0$ such that $F(x,t) > R t^2,$ $\forall \, |t|>t_R.$ Let $A(R,t):=\{x \in \mathbb{R}^N: t |v(x)|> t_{R} \},$ for $t>0.$ We have that
\begin{align}
\int_{\mathbb{R}^N}F(x,tv) \dx = \int_{ K_t}F(x,tv) \dx + \int_{ A(R,t)} F(x,tv) \dx \nonumber \\ 
\geq \int_{K_t }F(x,tv) \dx + R t^2 \int _{A(R,t)} v^2 \dx,\label{eq_geometri_ined}
\end{align}
where $K_t = (\mathbb{R}^N \setminus A(R,t) ) \cap \supp (v).$ Using Remark \ref{remarkao}--(ii), for each $t>0,$ we get that  
\begin{equation*}
|F(x,tv)| \leq C,\quad\text{for a.e. } x \in K_t,
\end{equation*}
where $C>0$ does not depend in $x$ and $t$. Consequently, for any $x \in \supp v,$ $F(x,tv)\mathcal{X}_{K_t} (x) \rightarrow 0,\quad \text{as }t\rightarrow \infty,$ where we have used that, for any $x \in \supp (v),$  $
\mathcal{X}_{\mathbb{R}^N \setminus A(R,t)} (x) \rightarrow \mathcal{X}_{\mathbb{R}^N \setminus \supp (v)}(x) = 0,\quad \text{as }t\rightarrow \infty,$ where $\mathcal{X}_{A}$ denotes the characteristic function of the set $A.$ Thus Dominated Convergence Theorem implies that the first integral in the right-hand side of inequality \eqref{eq_geometri_ined} goes to zero as $t$ goes to infinity. By the same reason, we also have
\begin{equation*}
\lim _{t \rightarrow \infty } \int_{A(R,t)} v^2 \dx = \lim _{t \rightarrow \infty } \int_{\mathbb{R}^N } v^2 \mathcal{X}_{A(R,t)}\dx = \int_{\mathbb{R}^N } v^2\mathcal{X}_{\{v \neq 0\}} \dx=\int_{\mathbb{R}^N } v^2 \dx
\end{equation*}
In particular, there exists a positive number $t_{0,R}$ such that
\begin{equation}\label{eq_finalmente}
\frac{1}{2}\int_{\mathbb{R}^N} v^2 \dx<\int_{A(R,t)} v^2\dx,\quad\forall \, t>t_{0,R}.
\end{equation}
Replacing \eqref{eq_finalmente} in \eqref{eq_geometri_ined} we have for $R$ sufficiently large,
\begin{align*}
I(tv)&=\frac{t^2}{2}\|v\|^2 _{V} -\frac{t^2}{2}\int_{\mathbb{R}^N}b(x) v^2\dx - \int_{\mathbb{R}^N} F(x,tv)\dx\\
&\leq \frac{1}{2}\left(\|v\|^2_{V} - R\|v\|^2_2\right)t^2 -  \int_{K_t }F(x,tv) \dx <0,\quad\text{for }t>t_{0,R}. \qedhere
\end{align*}
\end{proof}
\begin{remark}\label{r_minimax}
\begin{enumerate}[label=(\roman*)]
\item In view of Lemma \ref{geometry}, we define the set $\Gamma^1 _{I} = \{ \gamma \in C([0,1],  H^s _{V}(\mathbb{R} ^N) ) : \gamma(0)=0,\ \|\gamma(1)\|_V>r,\ I (\gamma(1) )<0 \},$ and $c_1(I)= \inf_{\gamma \in \Gamma_I } \sup_{t \in [0,1]} I(\gamma (t)),$ the usual minimax level. Thus we have $c_1(I)= c(I).$
\item If $f(x,t)\equiv f(t),$ the mountain pass geometry can be proved by replacing \eqref{posi_algum} by \eqref{posi_algum_auto}. In fact, let $\xi _R$ as in the proof of Lemma \ref{geometry} and define $\eta _R (x) = \xi _R (|x|).$ Then, as in \cite[Remark 2.8]{pala_vald_sch}, we have
\begin{align*}
\int_{\mathbb{R}^N} F(\eta_R ) \dx &= \int _{B_R (x_0)} F(t_0) \dx + \int_{B_{R+1}(x_0)\setminus B_R(x_0)} F(\eta _R ) \dx\\
&\geq F(t_0)|B_R|-|B_{R+1}\setminus B_R|\left(\max_{t \in [0,t_0]} |F(t)|\right).
\end{align*}
Thus there exist positive constants $C_1$ and $C_2$ such that for $R$ large,
\begin{equation*}
\int_{\mathbb{R}^N} F(\eta_R ) \dx\geq C_1 R^N - C_2 R^{N-1}>0,
\end{equation*}
The mountain pass geometry now follows as in the proof of Lemma \ref{geometry}.
\item If $f(x,t)$ satisfies \eqref{bem_def} and either \eqref{A-R}--\eqref{posi_algum} or \eqref{quadratic}; and additionally \eqref{lim_infinito}. Suppose also that $a(x)$ and $f(x,t)$ fulfills \eqref{V_sirakov}--\eqref{V_as} and \eqref{lim_infinito_infty}, respectively. Then the limiting functional $I_\infty$ has the mountain pass geometry. In fact, \eqref{lim_infinito_infty} together with \cite[Lemma 8.3]{paper1} implies that $\lambda_u(t):=u(\cdot /t),$ $t\geq 0,$ belongs to $\Gamma_{I_\infty},$ where $u \in H^s(\mathbb{R}^N)$ is such that
\begin{equation}\label{pathadminf}
\int_{\mathbb{R}^N} F_\infty(u) -\frac{V_\infty}{2}u^2 \dx> 0.
\end{equation}
As in Remark \ref{r_minimax}--(ii), we can see that there exists $\varphi_0 \in C_0 ^\infty (\mathbb{R}^N)$ satisfying \eqref{pathadminf} and
\begin{equation*}
I_\infty(\lambda_{\varphi_0} (t)) = \frac{1}{2}t^{N-2s}[\varphi_0]_s^2 - t^N \left[\int_{\mathbb{R}^N} F_\infty(\varphi_0) -\frac{V_\infty}{2}\varphi^2_0 \right] \rightarrow -\infty,\text{ as}\quad t \rightarrow \infty.
\end{equation*}
Moreover, $I(u) > 0$ if $\|u\|_{V} =r,$ for $r>0$ small enough (see proof of Lemma \ref{geometry}).
\item Under the assumptions of Lemma \ref{geometry}, and if $F(x,t) >0$ for a.e. $x \in \mathbb{R}^N$ and $t \neq 0,$ then, for any $u \in H^s_{V}(\mathbb{R}^N)\setminus \{0\},$ $\zeta(t) = t u$ belongs to $\Gamma_I.$ In fact, in the proof of Lemma \ref{geometry} replacing $v$ by $u$ and considering the same notations,
\begin{equation*}
\begin{aligned}
& \int_{\mathbb{R}^N}F(x,tu) \dx \geq R t^2 \int _{A(R,t)} u^2 \dx,  \\
& \lim _{t \rightarrow \infty } \int_{A(R,t)} u^2 \dx = \lim _{t \rightarrow \infty } \int_{\mathbb{R}^N } u^2 \mathcal{X}_{A(R,t)}\dx = \int_{\mathbb{R}^N } u^2\mathcal{X}_{\{u \neq 0\}} \dx=\int_{\mathbb{R}^N } u^2 \dx,
\end{aligned}
\end{equation*}
which enables us to proceed as in \eqref{eq_finalmente} to get for $R$ is large
\begin{equation*}
\varphi(t):=I(tu) \leq \frac{1}{2}\left(\|u\|^2_{V} - R\|u\|^2_2\right)t^2 \rightarrow - \infty,\text{ as }t\rightarrow \infty,
\end{equation*}
Moreover, assuming \eqref{increasing_total} we can infer that $\zeta(t)$ has a unique critical point.
\end{enumerate}
\end{remark}
From the previous results the existence of bounded PS sequence at the mountain pass level
\begin{proposition}\label{p_psbounded}
Assume that $a(x)\in L^1 _{\loca} (\mathbb{R}^N)$ fulfills \eqref{V_sirakov}--\eqref{V_be} and $f(x,t)$ satisfies either
\begin{enumerate}[label=(\roman*)]
\item \eqref{bem_def}--\eqref{posi_algum}; or
\item \eqref{posi_algum}--\eqref{dife_esti};
\end{enumerate}
Then there exists a bounded sequence $(u_k)$ such that $I(u_k) \rightarrow c(I)$ and $I' (u_k) \rightarrow 0.$
\end{proposition}
\begin{proof}
(i) In view of Lemma \ref{geometry}, the standard Mountain Pass Theorem implies the existence  of $(u_k)\subset H_{V} ^s (\mathbb{R}^N)$ such that $I(u_k) \rightarrow c(I)$ and $I'(u_k) \rightarrow0.$ For large $k,$ we have
\begin{align*}
c(I)+1+\| u_ k \|_{V}  &\geq I(u_k) - \frac{1}{\mu} I'(u_k) \cdot u_k \\
&=\left(\frac{1}{2} - \frac{1}{\mu} \right)\left(1-\frac{\|b(x)\|_{\beta}}{\mathcal{C}_{V} ^{(\beta)} } \right)  \|u_k\|_{V} ^2 - \int _{\mathbb{R}^N} F(x,u_k) - \frac{1}{\mu} f(x,u_k) u_k \dx\\
&\geq \left(\frac{1}{2} - \frac{1}{\mu} \right)\left(1-\frac{\|b(x)\|_{\beta}}{\mathcal{C}_{V} ^{(\beta)} } \right) \|u_k\|_{V} ^2,
\end{align*}
which implies that $(u_k)$ is bounded in $H^s _V (\mathbb{R}^N).$

(ii) The proof follows as in \cite[Lemma 2.5]{reinaldo} and \cite[Lemma 4.1]{dinglee}. By Lemma \ref{geometry}, applying a variant of the Mountain Pass Theorem, we obtain a Cerami sequence $(u_k)$ for $I$ at the level $c(I),$ precisely, $I(u_k) \rightarrow c(I)\text{ and }(1+\|u_k\|_{V})\|I'(u_k)\|_{\ast} \rightarrow 0,$ where $\| \cdot \|_{\ast}$ denote the usual norm of the dual of $H_{V}^s(\mathbb{R}^N).$ We claim that $(u_k)$ is bounded in $H^s _V (\mathbb{R}^N).$ Assume by contradiction that, up to a subsequence, $\|u_k\|_{V} \rightarrow \infty.$ Let $v_k = u_k / \| u_k \| _V.$ Thus
\begin{equation*}
\lim_{k \rightarrow \infty}\left[1-\int_{\mathbb{R}^N} \frac{f(x,u_k) }{\|u_k\|_{V}} v_k\dx - \frac{1}{\|u_k\|^2_{V}} \int_{\mathbb{R}^N}b(x)u^2_k\dx\right]=\lim_{k \rightarrow \infty}\left[ \frac{1}{\|u_k\|^2_{V}} I'(u_k) \cdot u_k \right] =0.
\end{equation*}
We can use an indirect arguments to prove that $ \lim_{k \rightarrow \infty} \int_{\mathbb{R}^N} f(x,u_k)v_k \|u_k\|^{-1}_{V} \dx = 0,
$ which by \eqref{V_be}, leads to the following contradiction,
\begin{equation*}
1=\lim_{k \rightarrow \infty} \frac{1}{\|u_k\|^2_{V}} \int_{\mathbb{R}^N}b(x)u^2_k\dx <\frac{1}{2}.
\end{equation*}
For $0\leq a<b \leq \infty ,$ defining $ \Omega_k (a,b) = \left\lbrace x \in \mathbb{R}^N : a \leq |u_k(x)| \leq b\right\rbrace,
$ we are going to prove that for $0<\varepsilon<1,$ there exist $k_\varepsilon, a_\varepsilon,b_\varepsilon$ such that
\begin{multline}
\int_{\mathbb{R}^N} \frac{f(x,u_k)}{\|u_k\|_{V}}v_k \dx = \int_{\Omega_k (0,a_\varepsilon)} \frac{f(x,u_k)}{\|u_k\|_{V}}v_k \dx \\+ \int_{\Omega_k (a_\varepsilon,b_\varepsilon)} \frac{f(x,u_k)}{\|u_k\|_{V}}v_k \dx + \int_{\Omega_k (b_\varepsilon,\infty)} \frac{f(x,u_k)}{\|u_k\|_{V}}v_k \dx <\varepsilon,\quad\forall \, k> k_\varepsilon.\label{eq_est_BPS}
\end{multline}
In order to do that, we first make some estimates involving $\mathcal{F}(x,t).$ Define $g(r) = \inf\left\lbrace   \mathcal{F}(x,t) : x \in \mathbb{R}^N,\ |t| > r \right\rbrace,$ which is positive and goes to infinity as $r\rightarrow \infty.$ Indeed, thanks to \eqref{dife_posi} and \eqref{dife_esti}, we have
\begin{equation*}
a_0 \mathcal{F}(x,t) \geq \left|\frac{f(x,t)}{t}\right| ^{p_0} > \left|2 \frac{F(x,t)}{t^2}\right|^{p_0},\quad\forall \,  |t|>R_0.
\end{equation*}
Consequently, by \eqref{quadratic}, we obtain that $\mathcal{F}(x,t) \rightarrow \infty,$ as $|t| \rightarrow \infty,$ uniformly in $x.$ Due to \eqref{dife_posi}, we also can define the positive number $m_a^b = \inf \left\lbrace \mathcal{F}(x,t) / t^2: x \in \mathbb{R}^N,\ a \leq |t| \leq b \right\rbrace.$ Using these notations, we see that there exists $k_0$ such that
\begin{align}
c(I)+1 & \geq I(u_k) - \frac{1}{2} I'(u_k) \cdot u_k \nonumber \\&=\int_{\Omega _k (0,a)} \mathcal{F}(x,u_k)\dx + \int_{\Omega _k (a,b)} \mathcal{F}(x,u_k)\dx+\int_{\Omega _k (b,\infty)} \mathcal{F}(x,u_k)\dx\nonumber\\
&\geq \int_{\Omega _k (0,a)} \mathcal{F}(x,u_k)\dx + m_a ^b \int_{\Omega _k (a,b)} u^2 _k \dx + g(b) |\Omega_k (b, \infty) |,\quad\forall \, k> k_0.\label{eq_unif_bound}
\end{align}
Inequality \eqref{eq_unif_bound} implies $\lim_{b \rightarrow \infty} |\Omega _k (b, \infty)|=0,\quad\text{uniformly in }k>k_0.$ Moreover, fixed $2<q\leq 2_s^\ast,$ we have
\begin{equation*}
\int_{\Omega _k (a,b)} |v_k|^q\dx \leq \left(\int _{\Omega _k (a,b) } |v_k| ^{2^\ast_s}\right)^{q/2_s^\ast} |\Omega_k (a,b)|^{(2_s^\ast -q)/2_s^\ast},
\end{equation*}
in particular,
\begin{equation}\label{eq_segunda_aux}
\lim _{b \rightarrow \infty} \int_{\Omega _k (a,b)} |v_k|^q\dx  =0,\quad\text{uniformly in }k>k_0.
\end{equation}
On the other hand, it follows that
\begin{equation}\label{eq_primeira_aux}
\int_{\Omega _k (a,b)} v_k^2 \dx = \frac{1}{\| u_k\|^2_{V}} \int_{\Omega _k (a,b)} u_k ^2\dx \leq \left(\frac{1}{\| u_k\|^2_{V}}\right) \left( \frac{1}{(c(I)+1)m_a^b}\right)\rightarrow 0,\text{ as }k\rightarrow \infty.
\end{equation}
We now pass to prove the estimate \eqref{eq_est_BPS}. By \eqref{quadratic}, there exists $a_\varepsilon >0$ such that 
\begin{equation*}
|f(x,t)|<\varepsilon |t|,\quad\text{a.e. }x \in \mathbb{R}^N,\text{ provided that } |t|< a_\varepsilon.
\end{equation*}
Thus, using \eqref{eq_primeira_aux} we have
\begin{equation*}
 \int_{\Omega_k (0,a_\varepsilon)} \frac{f(x,u_k)}{\|u_k\|_{V}}v_k \dx\leq\int_{\Omega_k (0,a_\varepsilon)\cap\{|u_k|>0\}} \frac{f(x,u_k)}{|u_k|} v^2_k \dx<\varepsilon /3,\quad\forall \, k> k^{(1)}_\varepsilon.
\end{equation*}
Taking $2q_0:=2p_0 / (p_0 - 1)$ and using \eqref{dife_esti}, \eqref{eq_segunda_aux} we have
\begin{align*}
\int_{\Omega_k (b_\varepsilon,\infty)} \frac{f(x,u_k)}{\|u_k\|_{V}}v_k \dx &\leq \int_{\Omega_k (b_\varepsilon,\infty)}\frac{f(x,u_k)}{|u_k|} v^2_k \dx\\
&\leq \left(a_0(c(I)+1) \right)^{1/p_0} \left(\int_{\Omega_k (b_\varepsilon,\infty)}|v_k|^{2q_0} \dx \right)^{1/q_0} < \varepsilon /3,\quad\forall \, k> k^{(2)}_\varepsilon.
\end{align*}
Using \eqref{quadratic} we get that $|f(x,u_k)| \leq C_\varepsilon |u_k|,\quad\text{a.e. }x \in \Omega_k (a_\varepsilon,b_\varepsilon),$ for $C_\varepsilon>0$ which does not depend on $k$ and $x$. Thus,
\begin{equation*}
\int_{\Omega_k (a_\varepsilon,b_\varepsilon)} \frac{f(x,u_k)}{\|u_k\|_{V}}v_k \dx \leq \int_{\Omega_k (a_\varepsilon,b_\varepsilon)} \frac{f(x,u_k)}{|u_k|} v^2_k \dx \leq C_\varepsilon  \int_{\Omega_k (a_\varepsilon,b_\varepsilon)} v^2_k \dx<\varepsilon /3,\quad\forall \, k> k^{(3)}_\varepsilon,
\end{equation*}
where $k^{(3)}_\varepsilon>k_0$ is obtained from \eqref{eq_primeira_aux}.
\end{proof}

\subsection{Behavior of weak profile decomposition convergence under nonlinearities} We now pass to describe the limit of the profile decomposition (Theorem \ref{teo_tinta_frac_sub}) for bounded sequences of the associated functional.
\begin{proposition}\label{lemmaconverge}
If $a(x)\equiv V(x)\in L ^1 _{\loca } (\mathbb{R}^N)$ and \eqref{bem_def}, \eqref{V_sirakov} hold, then for $(u_k) \subset H^s_{V} (\mathbb{R} ^N)$ a bounded sequence such that $u_k \rightarrow u$ in $L^p (\mathbb{R}^N),$ for some $p \in (2,2_s ^\ast),$ up to subsequence, we have
\begin{equation}\label{lema_conve}
\lim _{k \rightarrow \infty}\int _{\mathbb{R}^N} f(x,u_k) u_k \dx = \int _{\mathbb{R}^N} f(x,u)u \dx.
\end{equation}
Moreover, if $(v_k)$ is a bounded sequence in $H^s_{V}(\mathbb{R} ^N)$ with $u_k - v_k \rightarrow 0$ in $L^p (\mathbb{R}^N),$ for some $2<p<2_s ^\ast,$ then, up to a subsequence,
\begin{equation}\label{lema_conve_moreover}
\lim _ {k\rightarrow \infty} \int _{\mathbb{R}^N} F(x,u_k) - F(x,v_k) \dx=0.
\end{equation}
\end{proposition}
\begin{proof}
Note that $u_k \rightarrow u$ in $L ^q (\mathbb{R}^N)$ for all $q \in (2,2_ s ^\ast).$ This follows by a interpolation inequality, if $q<p$ then $\|u_k -u \| _q \leq \|u_k - u\|_2 ^{\theta} \|u_k - u\| _{p} ^{1- \theta}$ where $1/q = \theta / 2  + (1-\theta ) / p,$ and if $q>p$ then $\|u_k -u \| _q \leq \|u_k - u\|_p ^{\theta} \|u_k - u\| _{2_s ^\ast} ^{1- \theta} $ for $1/q = \theta / p  + (1-\theta ) / 2_s ^\ast.$ On the other hand, by \eqref{comp_loc} and Proposition \ref{welldef}, $u \in H^s _{V}(\mathbb{R}^N)$ and
\begin{equation*}
u_k (x) \rightarrow u(x) \text{ as }k \rightarrow \infty,\text{ a.e. }x \in\mathbb{R}^N\text{ and }|u_k(x)|,|u(x)|\leq h _\varepsilon (x)\text{ a.e. }x \in \mathbb{R}^N,\ k \in \mathbb{N},
\end{equation*}
for some $h _\varepsilon \in L ^{p_\varepsilon } (\mathbb{R} ^N).$ Now note that
\begin{equation*}
\int _{\mathbb{R}^N} |f(x,u_k)u_k - f(x,u)u| \dx \leq \int _{\mathbb{R}^N} |f(x,u_k) (u_k - u)| \dx + \int _{\mathbb{R}^N} |(f(x,u_k)- f(x,u)) u |\dx.
\end{equation*}
The first integral can be estimated by H\"{o}lder inequality as follows
\begin{equation*}
\int _{\mathbb{R}^N} |f(x,u_k) (u_k - u)| \dx \leq \varepsilon \left(\|u_k\|_2 \|u_k - u\| _2 + \|u_k\|^{2_s ^\ast -1} _{2_s ^\ast} \|u_k - u \| _{2_s ^\ast } \right)+C_\varepsilon \|u_k\| _{p_{\varepsilon} } ^{p _\varepsilon -1} \|u_k - u\| _{p_{\varepsilon}}.
\end{equation*}
For the second one, consider 
\begin{align*}
X^\varepsilon_k &:= \left\lbrace  x \in \mathbb{R}^N :\varepsilon (|u_k(x)| + |u_k(x)|^{2_s ^\ast -1}) \leq C_ \varepsilon |u_k (x)| ^{p_\varepsilon-1} \right\rbrace \text{ and}\\
X^\varepsilon &:= \left\lbrace  x \in \mathbb{R}^N :\varepsilon (|u(x)| + |u(x)|^{2_s ^\ast -1}) \leq C_ \varepsilon |u (x)| ^{p_\varepsilon-1} \right\rbrace .
\end{align*}
Thus
\begin{equation*}
\int _{ X^\varepsilon _k} |(f(x,u_k)- f(x,u)) u | \dx = \int _{\mathbb{R}^N} |(f(x,u_k)- f(x,u)) u | \mathcal{X} _{X^\varepsilon _k} \dx,
\end{equation*}
Since $ \mathcal{X} _{X^\varepsilon _k} (x)\rightarrow \mathcal{X} _{X^\varepsilon } (x)$ in $\mathbb{R}^N$ and $|(f(x,u_k)- f(x,u)) u  \mathcal{X} _{X^\varepsilon _k} | \leq 2 C_\varepsilon h _\varepsilon ^{p _\varepsilon} \in L^1 (\mathbb{R}^N),$ we may apply the Dominated Convergence Theorem to conclude
\begin{equation*}
\lim _{k \rightarrow \infty} \int _{X ^\varepsilon _k} |(f(x,u_k)- f(x,u)) u | \dx =0.
\end{equation*}
On the other way,
\begin{equation*}
\limsup _{k \rightarrow \infty}\int _{\mathbb{R}^N \setminus X^\varepsilon_k} |(f(x,u_k)- f(x,u)) u |\dx \leq C\varepsilon.
\end{equation*}
where $C>0$ does not depend in $\varepsilon$ and $k.$ Since $\varepsilon$ is arbitrary, \eqref{lema_conve} holds.\\
Now, let us prove \eqref{lema_conve_moreover}. Choose $(\bar{u} _k),(\bar{v}_k)$ in $C^\infty _0 (\mathbb{R^N})$ such that $\lim _{k \rightarrow \infty} \| \bar{u_k} - u_k \| _{V} = \lim _{k \rightarrow \infty} \| \bar{v_k} - v_k \| _{V}=0.$ Thus it suffices to prove that
\begin{equation}\label{lemma_suffices}
\lim _{k \rightarrow \infty }\int _{\mathbb{R}^N}  F(x,\bar{u_k}) - F (x,\bar{v_k})  \dx=0.
\end{equation}
Consider $E:= (C_0 (\mathbb{R}^N), \| \cdot \| _{p_\varepsilon })$ and $\beta : E \rightarrow \mathbb{R},$ given by $\beta (u) = \int _{\mathbb{R}^N} F(x,u)\dx$ with Gateaux derivative $\beta ' _G  (u) \cdot v= \int _{\mathbb{R}^N} f(x,u)v\dx.$ Thus, we may apply the Mean Value Theorem to get
\begin{equation}\label{MVT_lemma}
|\beta(u) - \beta(v)| \leq \sup _{w \in E,w \in [u,v]} \|\beta ' _G  (w) \| \ \| u - v\| _{p _\varepsilon}, \quad \forall \, u,v\in E,
\end{equation}
where $[u,v] = \{t u +(1-t)v:t \in [0,1]\}$. Since $(u_k),(v_k),(\bar{u} _k)$ and $(\bar{v}_k)$ belongs to a bounded set $B$ in $H^s _{V}(\mathbb{R}^N),$ and using $H^s_V (\mathbb{R}^N) \hookrightarrow L ^{p_ \varepsilon } (\mathbb{R}^N)$, we have that $B \cap E$ is bounded in $E.$ Thus $\beta ' _G$ is bounded in $B \cap E,$ which allows us to take $u=\bar{u}_k$ and $v=\bar{v}_k$ in \eqref{MVT_lemma} to get \eqref{lemma_suffices}.
\end{proof}
The next result is the nonlocal version of \cite[Lemma 5.1]{tintabook} and it is a generalization of the Brezis-Lieb Lemma.
\begin{proposition}\label{comp1}
Assume that $f(x,t)$ satisfies \eqref{bem_def} and \eqref{lim_infinito}. Let $(u_k)\subset H^s (\mathbb{R}^N) $ be a bounded sequence and $(w ^{(n)}) _{n \in \mathbb{N}_{0} }$ in $H^s(\mathbb{R}^N),$ given by the Theorem \ref{teo_tinta_frac_sub}. Then
\begin{equation*}
\lim _{k \rightarrow \infty}\int _{\mathbb{R}^N} F(x, u_k) \dx =\int _{\mathbb{R}^N}F(x,w^{(1)}) \dx + \sum _{n \in \mathbb{N}_0,n>1 } \int _{\mathbb{R} ^N} F _\mathcal{P} (x,w^{(n)}) \dx.
\end{equation*}
\end{proposition}
\begin{proof}
By the Proposition \ref{lemmaconverge} the functional $ \Phi (u) := \int _{\mathbb{R}^N} F(x,u)\dx,\quad u \in H^s (\mathbb{R}^N),$ is uniformly continuous in bounded sets of $L^p (\mathbb{R}^N),$ for any $2<p<2^{\ast } _s.$ Thus, by \eqref{seis.tres.sub} and \eqref{seis.quatro.sub},
\begin{equation*}
\lim _{k \rightarrow \infty} \left[ \Phi(u_k ) - \Phi \left(\sum _{n \in \mathbb{N} _0} w^{(n)} (\cdot - y_k ^{(n)} ) \right) \right] = 0.
\end{equation*}
By the uniform convergence in \eqref{seis.quatro.sub} we can reduce to the case $\mathbb{N}_0=\{ 1, \ldots, M\}.$ Thus taking $\Phi_\mathcal{P} (u) := \int_{\mathbb{R}^N} F _\mathcal{P}(x,u)\dx, \quad u \in H^s (\mathbb{R}^N),$ it follows from \eqref{lim_infinito} and Dominated Convergence Theorem that
\begin{equation*}
\lim_{k \rightarrow \infty} \left[ \sum _{n \in \mathbb{N} _0} \Phi \left(w^{(n)} (\cdot - y_k ^{(n)} ) \right)-\Phi (w^{(1)}) - \sum _{n \in \mathbb{N} _0 ,n>1} \Phi_\mathcal{P} (w^{(n)}) \right] = 0.
\end{equation*}
It remains to prove that
\begin{equation}\label{remains}
\lim _{k \rightarrow \infty} \left[ \Phi \left(\sum _{n \in \mathbb{N} _0} w^{(n)} (\cdot - y_k ^{(n)} ) \right) -\sum _{n \in \mathbb{N} _0} \Phi \left(w^{(n)} (\cdot - y_k ^{(n)}) \right)  \right] = 0.
\end{equation}
Since $\Phi$ is locally Lipschitz using a density argument, we can assume that $w^{(n)} \in C^\infty _0 (\mathbb{R}^N),$ for $n=1,\ldots,M.$ Consequently, from \eqref{seis.dois.sub}, 
\begin{equation*}
\supp (w^{(n)} (\cdot - y_k ^{(n)})) \cap  \supp (w^{(m)} (\cdot - y_k ^{(m)}) ) = \emptyset,\quad \text{ for }m \neq n\text{ and }k\text{ large enough,}
\end{equation*}
which implies that \eqref{remains}  holds, since for $k$ large enough,
\begin{align*}
\int _{\mathbb{R} ^N} F\left(x, \sum _{n \in \mathbb{N} _0} w^{(n)} (x - y_k ^{(n)} ) \right) \dx &=  \int _{ \bigcup _{n=1} ^M \supp (w^{(n)} (\cdot - y_k ^{(n)}) )} F \left(x, \sum _{m=1} ^M w^{(m)}(\cdot - y_k ^{(m)}) \right) \dx\\
&=\sum _{n=1} ^M \int _{\supp (w^{(n)} )} F(x+y_k ^{(n)}, w^{(n)} ) \dx \qedhere
\end{align*}
\end{proof}
\begin{corollary}\label{decomp_auto_lemma}
Let $(u_k)\subset H^s (\mathbb{R}^N) $ be a bounded sequence and $(w ^{(n)}) _{n \in \mathbb{N}_{0} }$ in $H^s(\mathbb{R}^N),$ given by Theorem \ref{teo_tinta_frac_sub}. If $f(x,t)$ is $\mathbb{Z}^N$--periodic and satisfies \eqref{bem_def},
\begin{equation}\label{lemaUM}
\lim _{k \rightarrow \infty}\int _{\mathbb{R}^N} F(x, u_k) \dx = \sum _{n \in \mathbb{N}_0} \int _{\mathbb{R} ^N} F (x,w^{(n)}) \dx.
\end{equation}
\end{corollary}
\begin{corollary}
Let $u_k \rightharpoonup u$ in $H^s (\mathbb{R}^N)$ and $F(x,t)$ as in Corollary \ref{decomp_auto_lemma} then, up to a subsequence,
\begin{equation*}
\lim _{k\rightarrow \infty}\int _{\mathbb{R}^N}  F(u_k) - F(u - u_k) - F(u)  \dx = 0.
\end{equation*}
\end{corollary}
\begin{proof}
Since $w^{(1)} = u,$ following the proof of Proposition \ref{comp1}, we obtain
\begin{equation}\label{lemaDOIS}
\lim _{k\rightarrow \infty} \int _{\mathbb{R}^N} F(u_k - u) \dx = \sum _{n\in \mathbb{N}_\ast,n>1} \int _{\mathbb{R}^N} F(w^{(n)}) \dx.
\end{equation}
Taking the difference between \eqref{lemaUM} and \eqref{lemaDOIS} we get the desired convergence.
\end{proof}

The following result, is a generalization of Fatou Lemma, or alternatively, that the functional $u \mapsto \int_{\mathbb{R}^N}V(x)u^2\dx$ is sequentially weakly lower semiconti\-nuous with respect to the profile decomposition of Theorem \ref{teo_tinta_frac_sub}. Moreover, it is a complement to Proposition \ref{comp1}.
\begin{proposition}\label{p_l2}
Suppose that $a(x)\equiv V(x)\geq 0$ and that \eqref{V_sirakov} holds true. Let $(u_k)$ be a bounded sequence in $H^s (\mathbb{R} ^N)$ and $(w ^{(n)}) _{n \in \mathbb{N}_{0} }$ given in Theorem \ref{teo_tinta_frac_sub}.
\begin{enumerate}[label=(\roman*)]
\item If \eqref{V_pe} holds, we have
\begin{equation*}
\liminf _{k\rightarrow \infty}\int_{\mathbb{R}^N}V(x)u_k^2\dx \geq \sum _{n \in \mathbb{N}_0}\int _{\mathbb{R}^N} V(x) |w^{(n)}|^2\dx.
\end{equation*}
\item Under \eqref{V_as} we obtain,
\begin{equation*}
\liminf _{k\rightarrow \infty}\int_{\mathbb{R}^N}V(x)u_k^2\dx \geq \int_{\mathbb{R}^N}V(x)|w^{(1)}|^2\dx+\sum _{n \in \mathbb{N}_0,n>1 }\int _{\mathbb{R}^N} V_\infty |w^{(n)}|^2\dx.
\end{equation*}
\end{enumerate}
\end{proposition}
\begin{proof}
We just prove the second inequality, the first one follows by a similar argument.
\begin{multline}\label{comp_l2}
\int_{\mathbb{R}^N} V(x) u^2_k\dx = \int _{\mathbb{R}^N} \left||V(x)|^{1/2} (u_k -w^{(1)} ) - |V_\infty|^{1/2} \sum _{n=2} ^m w^{(n)} (\cdot - y_k ^{(n)}) \right|^2\dx,\\
+ \int_{\mathbb{R}^N}V(x)|w^{(1)}|^2\dx+\sum _{n=2}^m\int _{\mathbb{R}^N} V_\infty |w^{(n)}|^2\dx+o(1),
\quad\forall \, m.
\end{multline}
where with the notation $a_k = o (b_k)$ we mean that $a_k / b_k \rightarrow 0.$ We proceed as in the proof of the iterated Brezis-Lieb Lemma \cite[Proposition 6.7]{tintacocompact}. We start by checking that \eqref{comp_l2} holds for $m=2.$ In fact, by Proposition \ref{welldef}, up to a subsequence, the classical Brezis-Lieb Lemma and assertion \eqref{seis.dois.sub} implies that
\begin{equation}\label{comp_l2_part1}
\int _{\mathbb{R}^N} V(x) u^2_k\dx =\int_{\mathbb{R}^N} V(x) |w^{(1)}|^2\dx +  \int_{\mathbb{R}^N}V(x) |u_k - w^{(1)}|^2 \dx+o(1),
\end{equation}
consequently and by the same reason,
\begin{multline}\label{comp_l2_part2}
\int_{\mathbb{R}^N} V(x) |u_k - w^{(1)}|^2 \dx = \\
\begin{aligned}
&\int _{\mathbb{R}^N} V(x+y_k ^{(2)}) |u_k (\cdot + y_k ^{(2)}) - w^{(1)} (\cdot + y_k ^{(2)}) | ^2 \dx\\
&+\int _{\mathbb{R}^N} \left||V(x+y_k ^{(2)})|^{1/2} \left(u_k (\cdot + y_k ^{(2)}) - w^{(1)}(\cdot + y_k ^{(2)}) \right) - |V_\infty w^{(2)}|^{1/2} \right|^2\dx\\
&+\int _{\mathbb{R}^N} V_\infty |w^{(2)}|^2 \dx +o(1).
\end{aligned}
\end{multline}
Replacing identity \eqref{comp_l2_part2} in \eqref{comp_l2_part1}, we obtain \eqref{comp_l2} for $m=2.$ We shall now prove that \eqref{comp_l2} holds for $m+1$ provided that it is true for $m.$ Indeed, arguing as above,
\begin{multline}\label{comp_l2_part3}
\int _{\mathbb{R}^N}\left||V(x)|^{1/2}(u_k - w^{(1)}) - V_\infty^{1/2}\sum _{n=2} ^m w^{(n)} (\cdot - y_k ^{(n)})\right|^2 \dx -\int _{\mathbb{R}^N} V_\infty |w^{(m+1)}| ^2 \dx\\= \int _{\mathbb{R}^N} \left||V(x)|^{1/2} \left(u_k  - w^{(1)} \right)  - V_\infty^{1/2} \sum _{n=2} ^{m+1} w^{(n)} (\cdot - y_k ^{(n)} ) \right|^2\dx + o(1).
\end{multline}
Applying the induction hypothesis in \eqref{comp_l2_part3} we obtain \eqref{comp_l2}.
\end{proof}
\subsection{Pohozaev Identity} We prove a Pohozaev type identity following the same argument of \cite[Section 4]{paper1} with some appropriated modifications. It complements some results in the present literature, namely: \cite[Theorem 2.3]{changnonexist}, \cite[Proposition 4.1]{chang-wang}, \cite[Proposition 4.3 ]{paper1} and \cite[Theorem 1.1]{rosoton}.
\begin{proposition}\label{pohozaev_id}
Assume $f(x,t)\equiv f(t)\in C^1 (\mathbb{R})$ and $a(x) \in C^1 (\mathbb{R}^N \setminus \mathcal{O}),$ where $\mathcal{O}$ is a finite set. Let $u\in \mathcal{D}^{s,2} (\mathbb{R}^N )$ be a weak solution of \eqref{P} such that $f(u)/(1+|u|)$ belongs to $L^{N/2s}_{\loca}(\mathbb{R}^N).$ If $F(u),$ $f(u)u,$ $a(x)u^2$ and $\left\langle \nabla a(x),x\right\rangle u^2$ belongs to $L^1(\mathbb{R}^N),$ then $u \in C^1 (\mathbb{R}^N \setminus \mathcal{O} )$ and
\begin{equation*}
\frac{N-2s}{2}\int _{\mathbb{R}^N} |(-\Delta)^{s/2} u|\dx +\frac{N}{2} \int _{\mathbb{R}^N} a(x) u^2 \dx + \frac{1}{2}\int _{\mathbb{R}^N} \left\langle \nabla a(x),x\right\rangle u^2 \dx = N \int _{\mathbb{R}^N} F(u) \dx.
\end{equation*}
\end{proposition}
\begin{proof}
We first prove the local regularity of $u$. For $x_0 \in \mathbb{R}^N\setminus \mathcal{O},$ $\overline{u} = u(\cdot+x_0)$ is a weak solution of $(-\Delta)^{s} \overline{u} +\overline{a}(x) \overline{u}= f(\overline{u})\text{ in } \mathbb{R}^N,$ where $\overline{a}(x) = a(x+x_0).$ Taking $r$ small enough, the ball $B^N_r$ does not contains any point of discontinuity of $\overline{a}(x)$ and so
	\begin{equation*}
	\frac{|g(\overline{u})|}{1+|\overline{u}|}\in L^{N/2s} (B^N_r),\quad \text{where}\quad g(\overline{u}):=f(\overline{u}) - \overline{a}(x)\overline{u}.
	\end{equation*}
	This enables us to proceed as in \cite[Proposition 4.2]{paper1}, to conclude that $u \in L^p (B_r ^N),$ for all $p\geq1.$ Moreover, since
	\begin{equation*}
	g(\overline{u})=f(\overline{u})-\overline{a}(x)\overline{u}=\left[\frac{f(\overline{u})}{1+|\overline{u}|}\sgn(\overline{u})- \overline{a}(x) \right]\overline{u} + \frac{f(\overline{u})}{1+|\overline{u}|},
	\end{equation*}
we can use the results of \cite{frac_niremberg} (see also \cite[Proposition 4.1]{paper1}) to conclude that there exist  $0<y_0,\ r_0<r$ with $B ^N _r \times [0,y_0] \subset B ^+ _r,$ and $\alpha \in (0,1),$ such that $\overline{w},\ \nabla _x \overline{w},\ y^{1-2s}\overline{w}_y \in C^{0, \alpha} (B ^N _{r_0} \times [0, y_0 ]),$ where $\overline{w}$ is the $s$-harmonic extension of $\overline{u}$ and $\nabla _x \overline{w} = (\overline{w}_{x_1} , \ldots , \overline{w}_{x_N}).$ Since $x_0$ is arbitrary,
	$
	w,\ \nabla_xw,\ y^{1-2s} w_y \in C(B ^N _{r}\setminus \mathcal{O} \times [0, y_0 ]  ),\quad \forall \,  r,y_0>0.
	$
	Consider $\xi \in C_0 ^\infty (\mathbb{R}:[0,1])$ such that $\xi(t)= 1,$ if $|t| \leq 1,$ $\xi(t)= 0,$ if $|t| \geq 2,$ and $|\xi ' (t)| \leq C,$ $\forall \, t \in \mathbb{R}$ for $C>0.$ Let $\mathcal{O}=\{x^{(1)},\ldots,x^{(l)} \},$ and $z^{(i)} = (x^{(i)},0),\ i=1,\ldots,l.$ For each $n=1,\ldots,$ define $\xi _n : \mathbb{R}^{N+1} \rightarrow \mathbb{R}$ by
	\begin{align*}
	\xi_n (z)=
	\left\{
	\begin{aligned}
	&\xi(|z|^2/n^2), \quad&\text{if }|z-z^{(i)}|^2 > 2/n^2,\\
	&1-\xi (n^2 |z-z^{(i)}|^2),\quad&\text{if }  |z-z^{(i)}|^2 \leq 2/n^2.
	\end{aligned}
	\right.
	\end{align*}
	Then, for $n$ large enough, $\xi _n \in C_0 ^\infty (\mathbb{R}^N)$ and verifies $|z||\nabla \xi _n (z)| \leq C,$ $\forall \, z\in \mathbb{R}^{N+1},$ for some $C>0.$ From now on the proof follows the same arguments of \cite[Proposition 4.3 ]{paper1}.\end{proof}
\begin{remark}
In previous the proof we have applied \cite[Theorem 2.15]{frac_niremberg} and for that it was crucial that $a(x)$ is a continuously differential function in $\mathbb{R}^N \setminus \mathcal{O}.$
\end{remark}
\begin{corollary}\label{pohozaev_standard}
Assume \eqref{bem_def} and $f(x,t) = f(t)\in C^1 (\mathbb{R}).$ Moreover, that $a(x)\equiv a_0,$ where $a_0$ is a positive constant. If $u \in H^s (\mathbb{R}^N)$ is a weak solution for \eqref{P}, then
\begin{equation*}
\int_{\mathbb{R}^N}F(u) -\frac{a_0}{2}u^2\dx = \frac{N-2s}{2N} \int_\mathbb{R^N} |(-\Delta )^{s/2}u|^2\dx
\end{equation*}
\end{corollary}
\begin{corollary}\label{cor_pohozaev_hardy}
Assume \eqref{bem_def_ast} and that $f(x,t) = f(t)\in C^1 (\mathbb{R}).$ If $u \in \mathcal{D}^{s,2} (\mathbb{R}^N)$ is a weak solution for \eqref{P}, then for $0<\lambda <\Lambda_{N,s}$ given by \eqref{par_hardy},
\begin{equation*}
\int_{\mathbb{R}^N}|(-\Delta )^{s/2}u|^2-\lambda |x|^{-2s}u^2\dx = \frac{2N}{N-2s} \int_\mathbb{R^N}F(u) \dx.
\end{equation*}
\end{corollary}
Next we have non-existence results, complementing the discussions made in \cite{mouhamednonexist}.
\begin{corollary} Assume $f(x,t)\equiv f(t)\in C^1(\mathbb{R}^N)$ and either one of the conditions are satisfied,
\begin{enumerate}[label=(\roman*)]
\item\label{COR_UM} $a(x) \in C^1 (\mathbb{R}^N \setminus \mathcal{O}),$ where $\mathcal{O}$ is a finite set, $2s a(x)+\left\langle \nabla a(x), x\right\rangle > 0$ for all $x$ in a non-zero measure domain and $2_s ^\ast F(t) \leq  f(t)t,$ for all $t\in \mathbb{R};$ or
\item\label{COR_DOIS} $a(x) \in C^1 (\mathbb{R}^N \setminus \mathcal{O}),$ where $\mathcal{O}$ is a finite set, $a(x)> 0,$ $\left\langle \nabla a(x),x\right\rangle > 0$ for all $x$ in a non-zero measure domain and there exists $0<\delta \leq 2,$ such that $\delta F(t) \geq f(t)t,$ for all $t\in \mathbb{R};$ or
\item\label{COR_TRES} $a(x) \equiv a_0>0$ constant and there exists $0\leq\delta \leq 2s/(N-2s),$ in a such way that $2_s ^\ast F(t) \leq f(t)t+\delta a_0 t^2,$  for all $t\in \mathbb{R};$
\item\label{COR_QUATRO} $a(x)\equiv 0$ and there exists $0<p<2_s^\ast$ such that $pF(t) \geq f(t)t$ for all $t\in \mathbb{R}.$
\end{enumerate}
If $u\in H^s (\mathbb{R}^N )$ is a weak solution of Eq. \eqref{P}, such that $F(u),$ $f(u)u,$ $a(x)u^2,$ $\left\langle \nabla a(x),x\right\rangle u^2$ belongs to $L^1(\mathbb{R}^N) $ and $f(u)/(1+|u|)$ belongs to $L^{N/2s}_{\loca}(\mathbb{R}^N),$ then $u\equiv0.$ 
\end{corollary}
\begin{proof}
\ref{COR_UM} Applying Proposition \ref{pohozaev_id}, we get
\begin{equation*}
\int_{\mathbb{R}^N} |(-\Delta )^{s/2}u|^2\dx+\frac{N}{N-2s}\int_{\mathbb{R}^N}a(x)u^2 \dx + \frac{1}{N-2s} \int _{\mathbb{R}^N}\left\langle \nabla a(x) , x\right\rangle u^2 \dx \leq \int _{\mathbb{R}^N} f(u)u\dx.
\end{equation*}
Using that $I'(u)\cdot u = 0,$ we obtain $u\equiv0,$ since
\begin{equation*}
\int_{\mathbb{R}^N} (2s a(x) + \left\langle \nabla a(x) , x\right\rangle )u^2 \dx \leq 0.
\end{equation*}
\ref{COR_DOIS} Using again Proposition \ref{pohozaev_id} we obtain that
\begin{equation*}
\frac{N-2s}{2N}\delta\int_{\mathbb{R}^N} |(-\Delta )^{s/2}u|^2\dx +\frac{\delta}{2} \int _{\mathbb{R}^N} a(x) u^2 \dx + \frac{\delta}{2N}\int _{\mathbb{R}^N} \left\langle \nabla a(x),x\right\rangle u^2 \dx \geq \int _{\mathbb{R}^N} f(u)u \dx,
\end{equation*}
which implies that
\begin{equation*}
\left(1 - \frac{N-2s}{2N} \delta \right)\int_{\mathbb{R}^N} |(-\Delta )^{s/2}u|^2\dx + \left(1-\frac{\delta}{2}\right)\int_{\mathbb{R}^N}a(x) u^2 \dx - \frac{\delta}{2N}\int_{\mathbb{R}^N}\left\langle \nabla a(x) , x\right\rangle  u^2 \dx \leq 0,
\end{equation*}
from this we get $u\equiv0.$

\ref{COR_TRES} Once more we can use Proposition \ref{pohozaev_id} to get
\begin{equation*}
\int_{\mathbb{R}^N} |(-\Delta )^{s/2}u|^2\dx+\frac{N}{N-2s}a_0\int_{\mathbb{R}^N}u^2 \dx \geq \int _{\mathbb{R}^N} f(u)u\dx,
\end{equation*}
which yields
\begin{equation*}
\left[ \frac{N-(1+\delta)(N-2s)}{N-2s} \right]  a_0 \int_{\mathbb{R}^N} u^2\dx \leq 0.
\end{equation*}
In particular $u\equiv0.$

\ref{COR_QUATRO} Proposition \ref{pohozaev_id} implies that $u\equiv0,$ since
\begin{equation*}
\int_{\mathbb{R}^N} |(-\Delta )^{s/2}u|^2\dx = 2_s ^\ast \int _{\mathbb{R}^N} F(u) \dx \geq \frac{2_s ^\ast}{p} \int _{\mathbb{R}^N} f(u)u \dx = \frac{2_s ^\ast}{p} \int_{\mathbb{R}^N} |(-\Delta )^{s/2}u|^2\dx.\qedhere
\end{equation*}
\end{proof}
\section{Proof of Theorem \ref{teoremao} }\label{s_proof_GS}
\begin{proof}
\ref{GS_UM} We use Theorem \ref{teo_tinta_frac_sub} which makes our argument easier then the one of \cite[Theorem 2.1]{reinaldo}. By Proposition \ref{p_psbounded}, there exists a bounded sequence $(u_k)$ such that $I(u_k) \rightarrow c(I)$ and $I'(u_k) \rightarrow 0.$ Using the profile decomposition provided by Theorem \ref{teo_tinta_frac_sub}, if we have $w^{(n)}= 0$ for all $n\in \mathbb{N}_0,$ then by assertion \eqref{seis.quatro.sub}, $u_k \rightarrow 0$ in $L^{p} (\mathbb{R}^N),$ for any $2<p<2_s ^{\ast }$ and by \eqref{seis.um.sub}, $u_k \rightharpoonup 0$ in $H_{V}^s(\mathbb{R}^N),$ up to a subsequence. Thus, by Proposition \ref{lemmaconverge},
\begin{equation}\label{eq_GS_primeira}
	\begin{aligned}
	& o(1)+c(I)=I(u_k) = \frac{1}{2}\| u_k\|^2 _{V} - \int_{\mathbb{R}^N} F(x,u_k) \dx = \frac{1}{2}\| u_k\|^2 _{V} + o(1),  \\
	& o(1)=I'(u_k)\cdot u_k = \|u_k \|^2 _{V} - \int _{\mathbb{R}^N} f(x,u_k) u_k \dx = \|u_k \|^2 _{V} + o(1),
	\end{aligned}
\end{equation}
which is a contradiction, with $c(I)>0.$ Thus, there must be at least one nonzero $w^{(n)}.$ Moreover, we have that each $w^{(n)}$ is a critical point of $I.$ In fact, up to a subsequence, we can take  $h ^{(n)}\in L ^{\sigma '} (\supp (\varphi)),$ $n\in \mathbb{N} _0,$ such that
\begin{equation}\label{dominada}
|u_k (x + y_k ^{(n)}) | \leq h ^{(n)}(x), \quad\text{a.e. }x \in \supp (\varphi),
\end{equation}
where $\sigma ' = \sigma / (\sigma -1)$ and $\varphi\in C^\infty _0 (\mathbb{R}^N),$ which can be done thanks to Proposition \ref{welldef}. Thus 
\begin{equation*}
	\begin{aligned}
	&|V(x + y_k ^{(n)}) u_k (x + y_k ^{(n)}) \varphi (x)|= |V(x) u_k (x + y_k ^{(n)}) \varphi (x)| \leq h ^{(n)}(x) |V(x) \varphi(x)| \in L ^1 (\supp (\varphi))  \\
	&V(x + y_k ^{(n)}) u_k (x + y_k ^{(n)}) \varphi (x) =V(x) u_k (x + y_k ^{(n)}) \varphi (x) \rightarrow V (x) w^{(n)} (x)\varphi (x), \quad\text{a.e. in } \mathbb{R}^N,
	\end{aligned}
\end{equation*}
which together with by the Dominated Convergence Theorem implies
\begin{align*}
\lim _{k \rightarrow \infty }(u_k ,\varphi (\cdot - y_k ^{(n)}) ) _{V} &= \lim_{k \rightarrow \infty} \left[  [u_k (\cdot + y_k ^{(n)} ), \varphi]_s + \int_{\mathbb{R}^N} V(x +y_k ^{(n)}) u_k (\cdot  +y_k ^{(n)}) \varphi (x) \dx \right]  \nonumber \\
&=[w^{(n)}, \varphi]_s + \int_{\mathbb{R}^N} V(x)w^{(n)}\varphi \dx. 
\end{align*}
By the same reason and \eqref{bem_def}, up to a subsequence we have,
\begin{equation*}
\lim _{k \rightarrow \infty} \int _{\mathbb{R}^N} f(x + y_k ^{(n)} , u_k (\cdot + y_k ^{(n)} ) ) \varphi \dx = \int _{\mathbb{R}^N} f(x, w^{(n)} )\varphi \dx.
\end{equation*}
Consequently we may pass the limit in
\begin{equation*}
I'(u_k)\cdot \varphi (\cdot - y_k ^{(n)}) = (u_k ,\varphi (\cdot - y_k ^{(n)}) ) _{V} - \int _{\mathbb{R}^N} f(x + y_k ^{(n)} , u_k (\cdot + y_k ^{(n)} ) ) \varphi \dx,
\end{equation*}
to conclude that $I'(w^{(n)})=0,$ for all $n\in \mathbb{N}_0$. In particular, we get that $\mathcal{G}_{\mathcal{S}} := \inf \{  I(u) : u\in H^s_{V}(\mathbb{R}^N)\setminus \{0 \},\ I'(u) = 0 \},$ is nonnegative. We are going to prove that is $\mathcal{G}_{\mathcal{S}}$ is attained and is positive. Let $(u_k)$ be a minimizing sequence for $\mathcal{G}_{\mathcal{S}},$ that is $I(u_k) \rightarrow \mathcal{G}_{\mathcal{S}}$ and $I'(u_k)=0.$ Arguing as in Proposition \ref{p_psbounded} we obtain that $(u_k)$ is bounded. Suppose by contradiction that $w^{(n)}=0$ for all $n\in \mathbb{N}_0.$ In this case we have $\mathcal{G}_{\mathcal{S}}>0,$ otherwise, if $\mathcal{G}_{\mathcal{S}}=0,$ then using \eqref{eq_GS_primeira} we would conclude that $\|u_k\|_{V} = o(1),$ and at the same time,
\begin{equation*}
\|u_k\|^2_{V} = \int _{\mathbb{R}^N}f(u_k)u_k \dx\leq \varepsilon(C_2 \|u_k\|^2_{V} + C_\ast\|u_k\|^{2_s ^\ast}_{V} ) + C_\varepsilon \|u_k\|^{p_\varepsilon}_{V},
\end{equation*}
where $\mathcal{C}_2,\ \mathcal{C}_{2_s^\ast},\ \mathcal{C}_{p_\varepsilon }>0$ are the constants given in Proposition \ref{welldef}. In particular, $ (1-\varepsilon \mathcal{C}_2) \leq \varepsilon \mathcal{C}_{2_s^\ast} \|u_k\|^{2_s ^\ast -2}_{V} + \mathcal{C}_{p_\varepsilon } \| u _k \|^{p_\varepsilon - 2}_{V},$ $\forall \, k\in \mathbb{N}, $ which, by taking $\varepsilon$ small enough, would lead to a contradiction with the fact that $\|u_k\|_{V} = o(1).$ In view of that, in any case, we can argue as above to conclude that there must be a nonzero $w^{(n_0)}$ which is a critical point of $I.$ From \eqref{seis.um.sub}, $u_k(x + y_k ^{(n_0)})\rightarrow w^{(n_0)}(x)$ a.e. in $\mathbb{R}^N.$ Thus
\begin{multline*}
\mathcal{G}_{\mathcal{S}} = \lim _{k \rightarrow\infty} I(u_k) = \lim _{k \rightarrow\infty}\int _{\mathbb{R}^N}  \mathcal{F} (x,u_k(\cdot + y_k ^{(n_0)})) \dx \\=  \liminf_{k\rightarrow \infty} \int _{\mathbb{R}^N}\mathcal{F} (x,u_k(\cdot + y_k ^{(n_0)})) \dx  \geq \int _{\mathbb{R}^N}\mathcal{F} (x,w^{(n_0)}) \dx = I(w^{(n_0)}),
\end{multline*}
where we have used \eqref{A-R} or \eqref{dife_posi} to ensure that $\mathcal{F} (x,u_k(\cdot + y_k ^{(n_0)})) = \mathcal{F}(x, u_k) \geq 0$ a.e. in $\mathbb{R}^N.$ Thus, once again using \eqref{A-R} or \eqref{dife_posi}, we can see that $\mathcal{G}_{\mathcal{S}}=I(w^{(n_0)})>0.$

\ref{GS_DOIS} From Proposition \ref{welldef}, the norm
\begin{equation*}
\vertiii{u}^2_{\lambda} = \int _{\mathbb{R}^N} |(-\Delta)^{s/2} u| ^2 -\lambda |x|^{-2s}u^2 \dx,\quad u \in \mathcal{D}^{s,2}(\mathbb{R}^N),\quad 0<\lambda<\Lambda _{N,s},
\end{equation*}
is equivalent with respect to the norm $[\ \cdot \ ]_s$ in $\mathcal{D}^{s,2}(\mathbb{R}^N).$ Let $(u_k)$ be a minimizing sequence for $\mathcal{I}_\lambda,$ and for each $k,$ let $u_k ^\ast$ be the Schwarz symmetrization of $u_k$. Applying the fractional Polya-Szeg\"{o} inequality (see \cite[Theorem 3]{beckner}), for each $k,$
\begin{equation*}
	\begin{aligned}
	& \int _{\mathbb{R}^N}\int _{\mathbb{R}^N}  \frac{ |u_k^\ast (x) - u_k ^\ast (y) |^2}{|x-y|^{N + 2s}}\dxdy \leq \int _{\mathbb{R}^N}\int _{\mathbb{R}^N} \frac{ |u_k(x) - u_k(y) |^2}{|x-y|^{N + 2s}}\dxdy,  \\
	& \int_{\mathbb{R}^N} F(u^\ast_k) \dx = \int_{\mathbb{R}^N} F(u_k) \dx.
	\end{aligned}
\end{equation*}
Thus $(u^\ast_k)\subset \mathcal{D}_\rad^{s,2}(\mathbb{R}^N)$ and is also a minimizing sequence for \eqref{min_crit}. Now observe that $\vertiii{\ \cdot \ }_\lambda$ is invariant with respect to the action of dilations given in Theorem \ref{teo_tinta_frac}, more precisely,
\begin{equation*}
\vertiii{u}_{\lambda} ^2 = \vertiii{\gamma ^{\frac{N-2s}{2} }u (\gamma ^j \cdot) }^2_{\lambda},\quad\forall \, u \in \mathcal{D}^{s,2}(\mathbb{R}^N),\ \gamma >1 \text{ and }j\in \mathbb{Z},
\end{equation*}
and satisfies the homogeneity property, $\vertiii{u(\cdot / \delta) }^2_\lambda = \delta^{N-2s} \vertiii{u}^2 _{\lambda},$ $u \in \mathcal{D}^{s,2}(\mathbb{R}^N),$ $\delta >0.$ In view of Remark \ref{r_radial_Dcompact} and Corollary \ref{cor_pohozaev_hardy}, the proof now follows the same argument of \cite[Theorem 3.4]{paper1}, replacing $[\ \cdot \ ]_s$ by $\vertiii{\ \cdot \ }_\lambda$.
\end{proof}
\begin{remark}
\begin{enumerate}[label=(\roman*)]
\item In the context of the proof of Theorem \ref{teoremao}--(i), if we assume in addition that $f(x,t)$ satisfies \eqref{increasing_total}, then $\mathcal{G}_{\mathcal{S}} = c(I)=I(w^{n_0})$ and $w^{(n_0)}$ is nonnegative. Indeed the truncation given in Remark \ref{r_nonnegative} satisfies the assumptions of Theorem \ref{teoremao}--(i), and we can apply the same argument there, to conclude that the ground state $w^{(n_0)}$ is nonnegative. Furthermore, Remark \ref{r_minimax}--(iv) guarantees that the path $\zeta (t) = t w^{(n_0)},$ $t \geq 0,$ belongs to $\Gamma_{I}$ and $c(I) \leq I(w^{(n_0)}).$ On the other hand, considering $(u_k)$ given in the beginning of the proof of Theorem \ref{teoremao}, by Corollary \ref{decomp_auto_lemma}, Remark \ref{assymp_defined}--(ii) and estimate \eqref{seis.tres.sub}, up to a subsequence, we have
\begin{equation*}
c(I)=\lim _{k \rightarrow \infty}\left[  \frac{1}{2}\|u_k\|_{V} ^2 - \int_{\mathbb{R}^N} F(x,u_k) \dx \right] \geq \sum _{n \in \mathbb{N} _0} I ( w^{(n)}).
\end{equation*}
Consequently, using \eqref{A-R} or \eqref{dife_posi} to get $I ( w^{(n)}) \geq 0 $, we conclude that $c(I) = \mathcal{G}_{\mathcal{S}}.$
\item If we consider the infimum \eqref{min_crit} defined over $\mathcal{D}_\rad^{s,2} (\mathbb{R}^N),$ by Remark \ref{r_radial_Dcompact} we can obtain concentration-compactness of the minimizing sequences as described in \cite[Theorem 3.4]{paper1}. More precisely, for any minimizing sequence $(u_k)$ of \eqref{min_crit}, there exists a sequence $(j_k)$ in $\mathbb{Z}$ such that the sequence $(\gamma^{-\frac{N-2s}{2}j_k } u_k(\gamma^{-j_k}\cdot ) )$ contains a convergent subsequence in $\mathcal{D}^{s,2}_\rad (\mathbb{R}^N),$ whose the limit is a minimizer of \eqref{min_crit} in $\mathcal{D}^{s,2}_\rad (\mathbb{R}^N).$
\item In the context of the proof of Theorem \ref{teoremao} (ii), assume that $F(t) \geq 0$ for all $t\geq 0.$ Since $\vertiii{|u_k|} _\lambda \leq \vertiii{u_k}_\lambda,$ without loss of generality we can assume that each $u_k$ is nonnegative. In this case, the obtained minimizer for \eqref{min_crit} is nonnegative.
\end{enumerate}
\end{remark}
\section{Proof of Theorem \ref{teo_GS_alt}}
\begin{proof}
As mentioned, we prove Theorem \ref{teo_GS_alt} by using the Nehari manifold method (see \cite{szulkin}). For convenience of the reader the proof will be divided into several steps.\\
\noindent \emph{Step 1.} For each $u \in H^s_{V} \setminus \{0 \} $ there exists a unique $\tau(u) >0$ such that $\tau (u) u \in \mathcal{N}$ and $\max_{t \geq 0} I (t u) = I (\tau (u)u).$ In particular $\mathcal{N} \neq \emptyset.$ To see that the function $h_u (t)= I (t u),$ $t>0,$ has a maximum point $t_u,$ we proceed in a similar way as in the Remark \ref{r_minimax}--(iv). Moreover, $h' (t_u) = 0,$ if and only if $t_u u $ belongs to $\mathcal{N}$ and
\begin{equation}\label{aux_nehari}
\|u\| _{V} ^2-\int _{\mathbb{R}^N}b(x) u^2 \dx = \frac{1}{t_u}\int_{\mathbb{R}^N} f (x, t_u u) u   \dx.
\end{equation}
By \eqref{increasing_total} the right-hand side of the above identity occurs at most one point. Thus there is a unique maximum point $\tau(u)=t_u$ for the function $h_u (t).$

\noindent \emph{Step 2.} The function $\tau : H^s_{V} \setminus \{0 \} \rightarrow (0,\infty)$ is continuous. Thus the map $\eta : H^s_{V} \setminus \{0 \}  \rightarrow \mathcal{N},$ defined by $\eta (u) =\tau (u)u$ is continuous and $\eta \big|_{\mathcal{S}}$ is a homeomorphism of the unit sphere $\mathcal{S}$ of $H^s _{V} (\mathbb{R}^N)$ in $\mathcal{N}.$

Assume that $u_n\rightarrow u$ in $H^s_{V} \setminus \{0 \}.$ Using that $F(x,t)>0$ and \eqref{A-R} we have $F(x,t) \geq C_1 |t|^\mu -C_2 t^2,$ $\text{for a.e. }x \in \mathbb{R}^N$ and $\forall \, t \in \mathbb{R}.$ Thus, from identity \eqref{aux_nehari} we obtain that
\begin{equation*}
\|u_n\| _{V} ^2-\int _{\mathbb{R}^N}b(x) u_n^2 \dx \geq C_1 |\tau(u_n)|^{\mu -2} \int_{\mathbb{R}^N}|u_n|^\mu \dx-C_2\|u_n\| _{V} ^2,\quad\forall \, n\in \mathbb{N},
\end{equation*}
that is, $(u_n)\subset L ^{\mu } (\mathbb{R}^N)$ with $\| u_n \| _{V} ^2 \geq C |\tau(u_n)| ^{\mu -2} \int_{\mathbb{R}^N}|u_n|^\mu \dx,$ $\forall \, n\in \mathbb{N}.$ Moreover, since $u \neq 0,$ $\|u_n\| > C>0,$ $\forall n.$ Thus $(\tau (u_n))$ is bounded. We next prove that any subsequence of $(\tau (u_n))$ has a convergent subsequence with the same limit $\tau(u),$ thus $\tau(u_n) \rightarrow \tau (u).$ It is clear that for a subsequence $\tau(u_n) \rightarrow t_0>0.$ In fact, using \eqref{bem_def}, \eqref{V_be} and \eqref{aux_nehari},
\begin{equation*}
\| u _n \| _{V} ^2 -\int _{\mathbb{R}^N}b(x) u_n^2 \dx \leq \varepsilon C \left( \| u_n \|^2 _{V} + \tau (u_n) ^{2 _s ^\ast -2} \| u_n \| ^{2 ^\ast _s} _{V}\right) + C _\varepsilon \tau(u_n) ^{p _\varepsilon -2} \| u_n \| _{V} ^{p _\varepsilon},\ \forall \, n\in \mathbb{N}.
\end{equation*}
From which, we obtain
\begin{equation}\label{nehari_estimate}
\left(1 -\varepsilon \mathcal{C}_2 - \frac{\|b(x)\|_\beta}{\mathcal{C}_{V} ^{(\beta)}}  \right)\| u _n \| _{V} ^2 \leq  \varepsilon \mathcal{C}_{2_s^\ast} \tau (u_n) ^{2 _s ^\ast -2} \| u_n \| ^{2 ^\ast _s} _{V}  + C _\varepsilon \mathcal{C}_{p_\varepsilon} \tau(u_n) ^{p _\varepsilon -2} \| u_n \| _{V} ^{p _\varepsilon},\ \forall \, n\in \mathbb{N},
\end{equation}
which implies $t_0 >0,$ by taking $\varepsilon$ small enough. Thus we may apply the Dominated Convergence Theorem in \eqref{aux_nehari} to conclude that $t_0 = \tau (u)$ and the continuity of the function $\tau.$ Using \eqref{aux_nehari} to compute $\tau (u/\|u\|_{V})$ we obtain that
\begin{equation*}
\|u\| _{V} ^2-\int _{\mathbb{R}^N}b(x) u^2 \dx = \frac{1}{\frac{\tau(u/\|u\| _{V})}{\|u\| _{V}}} \int_{\mathbb{R}^N} f\left(x,\frac{\tau(u/\|u\| _{V})}{\|u\| _{V}}\right) u \dx,
\end{equation*}
which by uniqueness gives $\tau (u/\|u\|_{V})= \tau (u)u.$ Consequently the inverse of $\eta$ is the retraction map given by $\varrho : \mathcal{N} \rightarrow \mathcal{S},\ \varrho (u)=u/\|u\| _{V}.$\\
\noindent \emph{Step 3.} $\mathcal{N}$ is away from the origin, that is, there is $R_{\mathcal{N}} >0$ such that $\|u\|_V > R_{\mathcal{N}} > 0$ if $u \in \mathcal{N}.$ Indeed, estimate \eqref{nehari_estimate} implies that
\begin{equation*}
1 -\varepsilon \mathcal{C}_2 - \frac{\|b(x)\|_\beta}{\mathcal{C}_{V} ^{(\beta)}}  \leq \varepsilon \mathcal{C}_{2_s^\ast} \| u\|_{V} ^{2_s ^\ast -2} +C _\varepsilon \mathcal{C}_{p_\varepsilon} \| u \| _{V} ^{p_\varepsilon -2},\quad \forall \, u \in \mathcal{N}.
\end{equation*}
\noindent\emph{Step 4.} For all $\zeta \in \Gamma_I$ we have that $\zeta ([0, \infty )) \cap \mathcal{N} \neq \emptyset.$ Let us suppose that this assertion is false, that is, there exists $\zeta_0 \in \Gamma _{I}$ which does not intercepts $\mathcal{N}$ in any point. Let $t_0>0$ such that $I (\zeta_0 (t_0)) <0$ and $\zeta_0 (t) \neq 0,$ for all $(0,t_0].$ We prove now that $\tau (\zeta (t))>1$ for all $t\in (0,t_0].$ In fact, by continuity, there is $\delta>0$ such that $\|\zeta _0 (t)\| < R_{\mathcal{N}},$ for all $t\in [0, \delta].$ We also have that $ \|\tau (\zeta_0 (t) ) \zeta_0 (t) \|_{V} > R_{\mathcal{N}},$ which implies $\tau (\zeta_0 (t)) >1,$ for all  $t\in (0, \delta].$ The continuity of $\tau (t)$ and the fact that $\zeta _0 (t) \not\in \mathcal{N},$ for all $t,$ allows us to choose $\delta = t_0.$ On the other hand, by \eqref{A-R} and \eqref{increasing_total},
\begin{align*}
h_{\zeta (t_0) } (t) & \geq \frac{t^2}{2}\left[ \| \zeta _0 (t_0)\|^2_V  - \int_{\mathbb{R}^N} b(x) |\zeta _0 (t_0)|^2 \dx - \frac{2}{\mu }\int_{\mathbb{R}^N} \frac{f(x,t \zeta_0 (t_0))}{t \zeta_0 (t_0)} |\zeta _0 (t_0)|^2 \dx \right]\\
& >\frac{t^2}{2} \left[\int_{\mathbb{R}^N} \frac{f(x,\tau(\zeta_0 (t_0)) \zeta_0 (t_0) )}{\tau(\zeta_0 (t_0)) \zeta_0 (t_0)}|\zeta _0 (t_0)|^2-\frac{f(x,t \zeta_0 (t_0))}{t \zeta_0 (t_0)} |\zeta _0 (t_0)|^2 \dx\right]\\
&>0,\quad\forall \,t\in (0,\tau (\zeta (t_0))].
\end{align*}
In particular, $0<h_{\zeta (t_0) } (1)=I (\zeta_0 (t_0)),$ which is a contradiction with the choice of $\zeta_0 (t_0).$

\noindent\emph{Step 5.} $ c_{\mathcal{N}  }(I)  = \bar{c}(I).$ In fact, since $\eta \big|_{\mathcal{S}}$ is a homeomorphism, we have $
\bar{c}(I) = \inf _{u \in H^s _V \setminus \{0\}} I (\tau(u) u) = \inf _{u \in \mathcal{S}} I (\tau(u) u)=  c_{\mathcal{N}  }(I). $

\noindent\emph{Step 6.} $\bar{c} (I ) = c(I ).$ Given $u \in H^s _V \setminus \{0\}$, define the path $\zeta(t) = t t_0 u,$ where $t_0>0$ is chosen in such way that $I(t_0 u)<0.$ Then, by Remark \ref{r_minimax}--(iv), it is easy to see that $\zeta \in \Gamma_I$ and $ \max_{t \geq 0}I(tu) = \max_{t \geq 0} I(\zeta(t)) \geq c(I).$ Consequently $c(I )\leq \bar{c} (I).$ On the other hand, given $\zeta \in \Gamma_I,$ there exists $t_0$ such that $\zeta (t_0)$ belongs to $\mathcal{N}.$ Thus, $ \max_{t \geq 0} I(\zeta(t)) \geq I(\zeta(t_0)) \geq c_{\mathcal{N}}(I)=\bar{c} (I ).$ Since $\zeta \in \Gamma_I$ is arbitrary, $c(I)\geq \bar{c} (I ).$
\end{proof}
\begin{remark}
In view of Remark \ref{r_nonnegative}, if $b(x) \equiv 0,$ then the radial ground state $u$ obtained in Theorem \ref{teo_GS_alt} can be considered as being nonnegative.
\end{remark}
\section{Proof of Theorem \ref{teo_sub} }
Before the proof of Theorem \ref{teo_sub}, to complement our discussion, we are going to compare the minimax level of limit functionals $I_\mathcal{P}$ and $I_\infty$ with the minimax level of the energy functional $I$ associated with Eq. \eqref{P}. Some arguments used to prove this result of comparison are used in the proof of Theorem \ref{teo_sub}.
\begin{proposition}\label{p_assumption_faz_sentido}
Assume that $f(x,t)$ satisfies either \eqref{bem_def}--\eqref{posi_algum}, \eqref{lim_infinito} or \eqref{posi_algum}--\eqref{dife_esti}, \eqref{lim_infinito}. Moreover, suppose that $b(x)\equiv 0,$ \eqref{V_pe}--\eqref{V_sirakov} and \eqref{increasing}. Then $c(I) \leq c(I _\mathcal{P}).$ Alternatively, if instead of the last set of hypothesis we assume that $V(x)\geq 0,$ $b(x)$ has compact support, \eqref{V_sirakov}--\eqref{V_as} and  \eqref{lim_infinito_infty}, then $c(I) \leq c(I _\infty ).$ Moreover, under these conditions, if we assume  \eqref{suf_cond_sub}, then \eqref{minimax_comper} and \eqref{minimax_comper'} holds true respectively for each considered case.
\end{proposition}
\begin{proof}
(i) Let $u\in H^s_{V}(\mathbb{R}^N)$ be a nonnegative (see Remark \ref{r_nonnegative}) nontrivial weak solution for $(-\Delta)^s u + V(x) u = f_\mathcal{P}(x,u),$ at the mountain pass level for $I_\mathcal{P}$, that is, $I_\mathcal{P}(u) = c(I_\mathcal{P}).$ For each $k,$ we define the path $\zeta _k (t) = t u (\cdot - y_k),\quad t \geq 0.$ where $(y_k)$ is taken such that $|y_k|\rightarrow\infty.$ The idea is to prove that
\begin{equation}\label{sake}
c(I) \leq \limsup _{k \rightarrow \infty }\max _{t \geq 0} I(\zeta _k (t) ) \leq \max _{t \geq 0} I _\mathcal{P} (t u)=c(I_\mathcal{P}).
\end{equation}
In fact, taking into account that $I$ and $I_\mathcal{P}$ are locally Lipschitz sets of $H^s_{V}(\mathbb{R}^N)$ (they are $C^1$ in $H^s _{V}(\mathbb{R}^N)$) and the following estimate
\begin{equation*}
\left| I( \zeta _k (t) ) - I _{\mathcal{P}} (tu) \right| \leq \int _{\mathbb{R}^N } \left| F(x + y_k,tu) - F_\mathcal{P} (x + y_k,tu) \right| \dx,
\end{equation*}
by using a density argument we get that $\lim _{k \rightarrow \infty } I( \zeta _k (t) ) = I _{\mathcal{P}} (tu),$ uniformly in compact sets of $\mathbb{R}.$ Consequently we may proceed as in \cite[Proposition 9.1]{paper1}. First note that
\begin{equation*}
\lim_{k \rightarrow \infty}\int _{\mathbb{R}^N} F(x+y_k,tu)\dx=\int _{\mathbb{R}^N}F_\mathcal{P}(x,tu) \dx,\quad\text{for each }t>0.
\end{equation*}
In particular,
\begin{equation*}
\int _{\mathbb{R}^N} F(x+y_k,u)\dx >0,\quad\text{ for }k\text{ large enough.}
\end{equation*}
Thus, using either \eqref{bem_def}--\eqref{posi_algum} or \eqref{posi_algum}--\eqref{dife_esti} and the arguments of Remark \ref{r_minimax}--(iv), we see that, for $k$ large enough, $ \zeta _k$ belongs to $\Gamma _{I}.$ As a consequence, there exists $t_k>0$ such that $I(\zeta _k (t_k) )= \max_{t \geq 0} I(\zeta _k (t))>0.$ We claim that $(t_k)$ is bounded. In fact, assume by contradiction that up to a subsequence $t_k \rightarrow \infty.$ Thus, by the arguments of Remark \ref{r_minimax}--(iv) we get
\begin{equation*}
I(\zeta _k (t_k) ) = \frac{t_k^2}{2}\|u\|^2 _{V} -\int_{\mathbb{R}^N} F(x+y_k,t_k u )\dx\rightarrow - \infty,\text{ as }t\rightarrow \infty,
\end{equation*}
which leads to a contradiction with the fact that $I(\zeta _k (t_k) ) > 0$ for all $k.$ Therefore, up to a subsequence, $t_k \rightarrow t_0,$ and thus $\lim_{k\rightarrow \infty}\max _{t \geq 0} I(\zeta _k (t_k))=I _\mathcal{P} ( t_0 u ),$ which leads to \eqref{sake}.

(ii). The second case is proved in a similar way. Let $w\in H^s_{V}(\mathbb{R}^N)=H^s(\mathbb{R}^N)$ be a nontrivial weak solution for the equation $(-\Delta)^s w + V_\infty w = f_\infty(w),$ at the mountain pass level, more precisely, $I_\mathcal{\infty}(w) = c(I_\mathcal{\infty}).$ For each $k,$ define the path $\lambda _k (t) = w \left((\cdot - y_k)/t\right) ,$ $t \geq 0.$ where $(y_k)$ is chosen in a such way that $|y_k|\rightarrow\infty.$ As before, we consider the estimate
\begin{multline*}
\left| I( \lambda _k (t) ) - I_\infty (w (\cdot / t )) \right| \\\leq \frac{t^N}{2}\int _{\mathbb{R}^N } \left|(  V(tx +y_k) -b(tx+y_k) )- V _\infty \right| w^2 \dx + t^N\int _{\mathbb{R}^N } \left| F(tx + y_k,w) - F_\infty (w) \right| \dx,
\end{multline*}
and the fact that $I$ and $I_\infty$ are Lipschitz in bounded sets of $H^s(\mathbb{R}^N)$ to obtain, by a density argument, that $\lim _{k \rightarrow \infty } I( \lambda _k (t) ) = I_\infty (w (\cdot / t )),$ uniformly in compact sets of $\mathbb{R}.$ We also have that the path $\lambda _k$ belongs to $\Gamma_I,$ for $k$ large enough. In fact, assuming the contrary, we would obtain $k_0$ and a sequence $t_n \rightarrow \infty$ such that $I(\lambda_{k_0} (t_n)) > 0,$ for all $n.$ On the other hand, we have that
\begin{equation*}
\lim_{n \rightarrow \infty} \int_{\mathbb{R}^N}F(t_nx+y_{k_0},w)- \frac{1}{2}\left[ V(t_nx+y_{k_0}) -b(t_nx+y_{k_0})\right] w^2\dx = \int_{\mathbb{R}^N}F_\infty(w) -\frac{1}{2}V_\infty w^2 \dx,
\end{equation*}
which leads to the contradiction $I(\lambda_{k_0} (t_n)) < 0,$ if $n$ large enough. Let $t_k>0$ such that $I(\lambda _k (t_k) )= \max_{t \geq 0} I(\lambda _k (t))>0.$ We claim that $(t_k)$ is bounded. On the contrary, 
\begin{multline*}
0<I(\lambda_k (t_{n_k}))\\= \frac{1}{2} t_{n_k}^{N-2s} [w]_s ^2 -t_{n_k}^N \left[\int_{\mathbb{R}^N}F(t_{n_k}x+y_k,w)- \frac{1}{2}(V(t_{n_k}x+y_k)-b(t_{n_k} x + y_k) w^2\dx\right]\\ \rightarrow - \infty,\text{ as }k\rightarrow \infty,
\end{multline*}
which is impossible. Thus, up to a subsequence, $t_k\rightarrow t_0$ and $\lim_{k\rightarrow \infty}\max _{t \geq 0} I(\lambda _k (t))=I _\infty (w(\cdot/t_0)).$ As a consequence, $c(I) \leq \lim_{k\rightarrow \infty}\max _{t \geq 0} I(\lambda _k (t_k)) \leq \max _{t \geq 0} I_\infty (w(\cdot / t)) = c(I_\infty),$ where we have used Corollary \ref{pohozaev_standard} to conclude that $t=1$ is the unique critical point of $ I_\infty (w(\cdot / t)).$

Now assume \eqref{suf_cond_sub}. Considering the above discussion, for each case respectively, we have for $k$ large enough.
\begin{equation*}
	\begin{aligned}
&c(I)\leq \max_{t\geq 0} I(\zeta_k(t)) = I(t_k u(\cdot - y_k) ) < I_\mathcal{P}(t_k u) \leq \max_{t\geq 0}I_\mathcal{P}(t_k u) = c(I_\mathcal{P}),\\
&c(I) \leq \max_{t\geq 0} I(\lambda_k(t)) = I(u((\cdot-y_k)/t_k)) < I_\infty(u(\cdot/t_k)) \leq \max_{t\geq 0}I_\infty(u(\cdot/t_k)) = c(I_\infty).\qedhere
\end{aligned}
\end{equation*}
\end{proof}
In order to prove our existence result without use the compactness conditions \eqref{minimax_comper} and \eqref{minimax_comper'}, we use the argument of \cite[proof of Theorem 1.2]{reinaldo}. Thus we use \cite[Theorem~2.3]{lins}, on 
the existence of a critical point of $I$ 
whenever the minimax level \eqref{minimax} is attained (see Remark \ref{r_minimax}--(i)).
\begin{proof}[Proof of Theorem \ref{teo_sub} completed]
In view of Lemma \ref{geometry} and Proposition \ref{p_psbounded} there exists a bounded sequence $(u_k)$ such that $I(u_k) \rightarrow c(I)$ and $I'(u_k) \rightarrow0,$ for both cases of this theorem. Let $(w ^{(n)})$ and $(y_k ^{(n)})$ be the sequences given in Theorem \ref{teo_tinta_frac_sub} for the sequence $(u_k).$ The underlying main idea to proof the concentration-compactness of Theorem \ref{teo_sub}, follows the same one of \cite[Theorem 3.6]{paper1} which we shall now describe: we prove that $w^{(n)}= 0$ for all $n\geq 2,$ which by assertions \eqref{seis.um.sub}, \eqref{seis.quatro.sub} and Proposition \ref{lemmaconverge} imply that $u_k \rightarrow w^{(1)}$ in $H_{V}^s(\mathbb{R}^N),$ up to a subsequence. To this end, we argue by contradiction and assume the existence of at least one $w^{(n_0)} \neq 0,$ $n_0 \geq 2.$

\ref{teo_sub_UM} In view of Remark \ref{assymp_defined}--(ii), by Proposition \ref{comp1} and \eqref{seis.tres.sub}, up to a subsequence, we have
\begin{equation}\label{teo_sub_minimax_est_UM}
c(I)=\lim _{k \rightarrow \infty}\left[  \frac{1}{2}\|u_k\|_{V} ^2 - \int_{\mathbb{R}^N} F(x,u_k) \dx \right] \geq I(w^{(1)})+\sum _{n \in \mathbb{N} _0,n >1} I_\mathcal{P} ( w^{(n)}).
\end{equation}
We claim that the terms on the right-hand side of \eqref{teo_sub_minimax_est_UM} are nonnegative. Indeed, following as in the proof of Theorem \ref{teoremao} $w^{(1)}$ and $w^{(n)},$ $n\geq 2,$ are critical points for $I$ and $I_\mathcal{P},$ respectively. In view of that, \eqref{A-R} or \eqref{dife_posi} implies that $I(w^{(1)}) \geq 0$ and $I_\mathcal{P}(w^{(n)}) \geq 0,$ $n\geq 2,$ respectively. On the other hand, Remark \ref{r_minimax}--(iv) guarantees that $\zeta^{(n_0)}(t)=tw^{(n_0)}\in \Gamma_{I_\mathcal{P}}$ and $c(I_\mathcal{P})<  I_\mathcal{P}(w^{(n_0)}).$ This, together with \eqref{teo_sub_minimax_est_UM} and \eqref{minimax_comper} leads to a contradiction.

\ref{teo_sub_DOIS} Following the proof of Theorem \ref{teo_tinta_frac_sub}, we can replace $\| \cdot\|$ by the equivalent norm $\| \cdot \|_{V_\infty}$ in assertions \eqref{seis.um.sub}--\eqref{seis.quatro.sub}. Consequently, by \eqref{seis.tres.sub}, Propositions \ref{comp1} and \ref{p_l2}, up to a subsequence, 
\begin{equation}
c(I)\geq\lim _{k \rightarrow \infty}\left[  \frac{1}{2}\|u_k\|_{V} ^2 -\int_{\mathbb{R}^N} b(x) u^2_k\dx- \int_{\mathbb{R}^N} F(x,u_k) \dx \right]\geq I(w^{(1)}) +\sum _{n \in \mathbb{N} _0,n >1} I_{\infty } ( w^{(n)}) \label{liminf}.
\end{equation}
Thus, it suffices to prove that the right-hand side of \eqref{liminf} is non-negative and $I_\infty(w^{(n)}) \geq c(I _\infty),$ $\forall \, n \geq 2.$ In fact, $c(I) \geq I(w^{(n_0)}) \geq c(I _\infty),$ which leads to a contradiction with \eqref{minimax_comper}. To do this, we prove that  $w^{(1)}$ and $w^{(n)}$ are critical points for $I$ and $I_\infty$ respectively, $n\geq 2$. Let $\varphi \in C^\infty _0 (\mathbb{R}^N)$ and $h^{(n)}\in L ^{2^\ast_s -1}(\supp (\varphi))$ as in \eqref{dominada}. By \eqref{V_as} and \eqref{seis.dois.sub}, there exists $k_0=k_0 ( \varphi)$ such that $V(x+y_k ^{(n)}) < 1 + V_\infty,$ $\forall \, k>k_0,$ $x \in \supp (\varphi)$ and $n \geq 2.$
Thus,
\begin{equation*}
	\begin{aligned}
	&|V(x + y_k ^{(n)}) u_k (x + y_k ^{(n)}) \varphi (x)| \leq (\varepsilon + V_\infty) h ^{(n)}(x) |\varphi(x)| \in L ^1 (\supp (\varphi)), \text{ for }k>k_0,  \\
	&V(x + y_k ^{(n)}) u_k (x + y_k ^{(n)}) \varphi (x) \rightarrow V_\infty w^{(n)} (x)\varphi(x)\quad \text{a.e. in } \mathbb{R}^N,
	\end{aligned}
\end{equation*}
which together with the Dominated Convergence Theorem implies
\begin{align*}
\lim _{k \rightarrow \infty }(u_k ,\varphi (\cdot - y_k ^{(n)}) ) _{V} &= \lim_{k \rightarrow \infty} \left[  [u_k (\cdot + y_k ^{(n)} ), \varphi]_s + \int_{\mathbb{R}^N} V(x +y_k ^{(n)}) u_k (\cdot +y_k ^{(n)}) \varphi (x) \dx \right]  \nonumber \\
&=[w^{(n)}, \varphi]_s + \int_{\mathbb{R}^N} V_\infty w^{(n)} (x) \varphi (x) \dx. 
\end{align*}
And for the same reason,
\begin{equation*}
\lim _{k \rightarrow \infty} \int _{\mathbb{R}^N} f(x + y_k ^{(n)} , u_k (\cdot + y_k ^{(n)} ) ) \varphi \dx = \int _{\mathbb{R}^N} f_\infty ( w^{(n)} )\varphi \dx.
\end{equation*}
Consequently, taking the limit in
\begin{equation*}
I'(u_k)\cdot \varphi (\cdot - y_k ^{(n)}) = (u_k ,\varphi (\cdot - y_k ^{(n)}) ) _{V} - \int _{\mathbb{R}^N} f(x + y_k ^{(n)} , u_k (\cdot + y_k ^{(n)} ) ) \varphi \dx,
\end{equation*}
we deduce that $I'(w^{(1)})=0$ and $I'_\infty(w^{(n)})=0,$ $n\geq 2.$ Using \eqref{A-R} or \eqref{dife_posi} we get that $I(w^{(1)})\geq 0$ and $I_\infty(w^{(n)})\geq0,$ $n\geq 2.$ Finally, define $\lambda^{(n_0)}(t) = w^{(n_0)} (\cdot / t),$ $t\geq 0.$ By Corollary \ref{pohozaev_standard},
\begin{equation*}
I _\infty (\lambda^{(n_0)}(t)) = \frac{1}{2} t^{N-2s} [w^{(n_0)}]_s ^2 - t^N \left[ \int _{\mathbb{R}^N} F _{\infty} (w^{(n_0)}) - \frac{V _\infty }{2} |w^{(n_0)}|^2 \dx \right] \rightarrow - \infty,\text{ as }t\rightarrow \infty,
\end{equation*}
which together with Remark \ref{assymp_defined} implies that $\lambda^{(n_0)}\in \Gamma _{I_\infty}.$ Corollary \ref{pohozaev_standard} gives that $t=1$ is the unique critical point of $I _\infty (\lambda^{(n_0)}(t)).$ Thus, $c(I _\infty) < \max _{t \geq 0} I_\infty (\lambda^{(n_0)}(t)) = I_\infty (w^{(n_0)}),$ a contradiction.

\ref{teo_sub_TRES} Finally, assume \eqref{replace_improved} instead of \eqref{minimax_comper} and \eqref{minimax_comper'}. Consider the existence of $w^{(n_0)}\neq 0,$ $n_0 \in \mathbb{N}_0,$ and the paths $\zeta^{(n_0)}$ and $\lambda^{(n_0)}$ as above. Taking account the above discussion, by estimates \eqref{teo_sub_minimax_est_UM} and \eqref{liminf}, for each case we have
\begin{equation*}
	\begin{aligned}
&c(I)\leq \max_{t\geq 0} I(\zeta^{(n_0)}(t)) \leq \max_{t\geq 0} I_\mathcal{P}(\zeta^{(n_0)}(t) )= I_\mathcal{P}(w^{(n_0)}) \leq c(I),\\
&c(I)\leq \max_{t\geq 0} I(\lambda^{(n_0)}(t)) \leq \max_{t\geq 0} I_\infty(\lambda^{(n_0)}(t) ) = I_\infty (w^{(n_0)}) \leq c(I),
	\end{aligned}
\end{equation*}
where we have used \eqref{replace_improved} to ensure that the paths $\zeta^{(n_0)}$ and $\lambda^{(n_0)}$ belongs to $\Gamma_I.$ Thus, we have that the minimax level $c(I)$ is attained and we can apply \cite[Theorem~2.3]{lins} to obtain the existence of a critical point $u$ for $I_\lambda$ with $I_\lambda(u)=c(I_\lambda).$ If there is no $w^{(n)} \neq 0,$ $n \in \mathbb{N}_0,$ (which is the case where strict inequalities occurs) we can  obtain that $u_k \rightarrow w^{(1)},$ up to a subsequence.
\end{proof}
\section{Proof of Theorem \ref{teo_crit}}\label{s_proof_crit} 
\begin{proof}
The proof will be divided into three steps. We first assume \eqref{suficient_ast} and \eqref{suficient_ast_0}.
(i) We can proceed analogously to the proof of Lemma \ref{geometry}, to see that there exists a sequence $(u_k)$ in $\mathcal{D}^{s,2}(\mathbb{R}^N)$ such that $I _\ast(u_k) \rightarrow c(I_\ast)>0$ and $I' _\ast (u_k) \rightarrow 0.$ Let $(w ^{(n)}),\  (y_k ^{(n)}),\ (j _k ^{(n)})$ be the sequences given by Theorem \ref{teo_tinta_frac} and define the set
\begin{equation*}
\mathbb{N} _{\sharp} = \left\lbrace n \in \mathbb{N} _\ast\setminus\{1\}:|\gamma ^{j_k ^{(n)}}y_k ^{(n)}|\text{ is bounded} \right\rbrace .
\end{equation*}
Passing to a subsequence and using a diagonal argument if necessary, we may assume that each sequence $(\gamma ^{j_k ^{(n)}}y_k ^{(n)}),$ $n \in \mathbb{N} _{\sharp},$ converges with $a^{(n)} := \lim _{k\rightarrow \infty}\gamma ^{j_k ^{(n)}}y_k ^{(n)},$ $n \in \mathbb{N} _{\sharp}.$

(ii) Now we shall prove the following estimate, up to a subsequence
\begin{multline}\label{reasonbefore}
\limsup_k \|u_k \| _V ^2  \geq \|w^{(1)}\|_V ^2 \\+ \sum _{n \in \mathbb{N}_\ast \setminus \mathbb{N} _{\sharp}} [w^{(n)}]_s^2 + \sum _{n \in \mathbb{N} _+\cap\mathbb{N} _{\sharp}} \|w^{(n)}\| _{V_+(\cdot + a^{(n)}-a_\ast) }^2+\sum _{n \in \mathbb{N} _{-}\cap\mathbb{N} _{\sharp}} \|w^{(n)}\| _{V_{-}(\cdot + a^{(n)}-a_\ast) }^2,
\end{multline}
In order to prove this, first consider the operator 
\begin{equation*}
d_k ^{(n)} u = \gamma ^{\frac{N-2s}{2} j_k ^{(n)} } u(\gamma^{j _k ^{(n)}} ( \cdot - y_k^{(n)} ) ),\quad u\in \mathcal{D}^{s,2}(\mathbb{R}^N),\ n \in \mathbb{N}_\ast.
\end{equation*}
For each $n \in \mathbb{N}_\ast,$ let $(\varphi_j ^{(n)})$ in $C^\infty _0 (\mathbb{R}^N)$ such that $\varphi_j ^{(n)} \rightarrow w^{(n)}$ in $\mathcal{D}^{s,2} (\mathbb{R}^N).$ Evaluating
\begin{equation*}
\left\|u_k - \sum _{n \in M_\ast} d_k ^{(n)} \varphi_j ^{(n)} \right\|^2 _V \geq 0,
\end{equation*}
in a finite subset $M_\ast = \{1,\ldots,M\}$ of $\mathbb{N}_\ast,$ we have
\begin{equation}\label{rb_passlimit_ref}
\|u_k\|_V ^2 \geq 2 \sum_{n \in M_\ast} (u_k , d_k ^{(n)} \varphi _j ^{(n)} )_V - \sum_{n \in M_\ast} \|d_k ^{(n)} \varphi _j ^{(n)}\|^2_V.
\end{equation}
We are now going to study the limit in \eqref{rb_passlimit_ref}. Taking
\begin{equation*}
v_k ^{(n)} :=d_k ^{(n)} u_k = \gamma ^{-\frac{N-2s}{2}j_k ^{(n)}}u_k (\gamma ^{- j_k ^{(n)}} \cdot + y_k^{(n)} ),
\end{equation*}
we have
\begin{align*}
(u_k,d_k ^{(n)}\varphi _j ^{(n)}  )_V &= [v_k ^{(n)} ,\varphi _j ^{(n)} ]_s + \int_{\mathbb{R}^N} \gamma ^{-2s j_k ^{(n)}} V(\gamma ^{-j_k ^{(n)}}( (x +y_k ^{(n)}) +a_ \ast )) v_k ^{(n)}(\cdot + a_ \ast) \varphi _j ^{(n)}(\cdot + a_ \ast) \dx,\\
\| d_k ^{(n)}\varphi _j ^{(n)}\|_V ^2 &= [\varphi _j ^{(n)}]_s ^2 + \int_{\mathbb{R}^N}\gamma ^{-2sj_k ^{(n)}} V(\gamma ^{-j_k ^{(n)}}( (x +y_k ^{(n)}) +a_ \ast )) |\varphi _j ^{(n)}(\cdot + a_ \ast)|^2 \dx.
\end{align*}
Fixed $j,$ using \eqref{V_conve_seq_crit}, up to a subsequence we have
\begin{equation}\label{limite_aux_um}
\lim_{k \rightarrow \infty} (u_k,d_k ^{(n)}\varphi _j ^{(n)} )_V =[w^{(n)} ,  \varphi _j ^{(n)}]_s \\ \quad\text{ and }\quad\lim_{k \rightarrow \infty}\| d_k ^{(n)}\varphi _j ^{(n)} \|_V ^2 = [ \varphi _j ^{(n)}]^2_s,
\end{equation}
provided that $n \notin \mathbb{N}_{\sharp}$ (this is the case when $n \in \mathbb{N}_0$). Similarly, up to a subsequence, by \eqref{V_limits_ast},
\begin{equation}\label{limite_aux_dois}
\begin{aligned}
\lim_{k \rightarrow \infty} (u_k,d_k ^{(n)}\varphi _j ^{(n)})_V &=( w^{(n)}, \varphi _j ^{(n)})_{V_\kappa(\cdot + a^{(n)}-a_\ast)} \\
\lim_{k \rightarrow \infty}\| d_k ^{(n)}\varphi _j ^{(n)} \|_V ^2 &= \| \varphi _j ^{(n)}\|^2_{V_\kappa(\cdot + a^{(n)}-a_\ast)},
\end{aligned}
\end{equation}
where $\kappa =+,-,$ whenever $n \in \mathbb{N} _+\cap\mathbb{N} _{\sharp}$ or $\mathbb{N} _{-}\cap\mathbb{N} _{\sharp},$ respectively.  Since $\mathbb{N}_\ast \setminus \{1\}= (\mathbb{N}_\ast \setminus \mathbb{N}_{\sharp}) \dot{\cup} \left[ (\mathbb{N}_{+} \cap \mathbb{N}_{\sharp} ) \dot{\cup } (\mathbb{N}_{-} \cap \mathbb{N}_{\sharp} ) \right],$ up to a subsequence, we can apply the limits \eqref{limite_aux_um} and \eqref{limite_aux_dois} in \eqref{rb_passlimit_ref} to get
\begin{multline}\label{almost_finish_ineq}
\limsup_k\|u_k\|^2_V \geq \|w^{(1)}\|_V ^2 + \sum_{n \in M_\ast \cap\mathbb{N} _+\cap\mathbb{N} _{\sharp}  }2(w^{(n)} , \varphi _j ^{(n)})_{V_+ (\cdot + a^{(n)}-a_\ast)} - \|\varphi _j ^{(n)}\|_{V_+ (\cdot + a^{(n)}-a_\ast)}^2\\+\sum_{n \in M_\ast \cap \mathbb{N} _{-}\cap\mathbb{N} _{\sharp} }2(w^{(n)} , \varphi _j ^{(n)})_{V_{-} (\cdot + a^{(n)}-a_\ast)} - \|\varphi _j ^{(n)}\|_{V_{-} (\cdot + a^{(n)}-a_\ast)}^2 \\+\sum_{ n \in M_\ast \setminus \mathbb{N} _{\sharp}} 2 [w^{(n)} ,\varphi _j ^{(n)} ]_s - [\varphi _j ^{(n)}]^2_s.
\end{multline}
Since $\| \cdot \|_{V_{+}}$ and $\| \cdot \|_{V_{-}}$ are equivalent to $[\ \cdot \ ]_s$ in $\mathcal{D}^{s,2}(\mathbb{R}^N)$ we can take the limit in $j$ in \eqref{almost_finish_ineq} and use the arbitrariness of choice for $M$ to obtain \eqref{reasonbefore}.

(iii) If $w^{(n)}=0$ for all $n\geq 2,$ then $u_k \rightarrow w^{(1)}$ in $\mathcal{D}^{s,2}(\mathbb{R}^N),$ with $w^{(1)}$ being a critical point of $I_\ast.$ Let us argue by contradiction and assume the existence of $w^{(n_0)}\neq 0,$ with $n_0 \geq2 $. By \cite[Proposition 7.1]{paper1} and estimate \eqref{reasonbefore}, up to a subsequence, we have
\begin{equation}\label{crit-right-hand}
c(I_\ast) \geq I_\ast (w^{(1)})+\sum _{n \in \mathbb{N}_\ast \setminus \mathbb{N} _{\sharp}}I_0 (w^{(n)}) + \sum _{n \in \mathbb{N} _+\cap\mathbb{N} _{\sharp}}I^{(n)}_{+}(w^{(n)}) +\sum _{n \in \mathbb{N} _{-}\cap\mathbb{N} _{\sharp}} I^{(n)}_{-}(w^{(n)}),
\end{equation}
where
\begin{equation*}
I_\pm ^{(n)}(u) = \frac{1}{2}\| u \|^2 _{V_\pm(\cdot + a^{(n)}-a_\ast)}  - \int_{\mathbb{R}^N} F_\pm (u)\dx\quad\text{and}\quad I_0(u) = \frac{1}{2}[u]_s^2 - \int_{\mathbb{R}^N} F_0 (u)\dx,\quad u \in \mathcal{D}^{s,2}(\mathbb{R}^N).
\end{equation*}
As before, we prove that each $w^{(n)}$ is a critical point for the functionals in the respective index of the sums in \eqref{crit-right-hand}, and as a consequence of \eqref{A-R}, the right-hand side of \eqref{crit-right-hand} is non-negative. In the next step we obtain that $c(I_\ast ) < I^{(n)}_\kappa (w^{(n)})$ in the correspondent index, which leads to a contradiction with estimate \eqref{crit-right-hand}. In fact, given $\varphi$ in $C^\infty _0 (\mathbb{R}^N),$ as in the proof of \eqref{reasonbefore},
\begin{equation*}
\lim _{k \rightarrow \infty} (u_k,d_k ^{(n)} \varphi )_V = [w^{(n)},\varphi]_s\quad\text{ and }\quad\lim _{k \rightarrow \infty}(u_k,d_k ^{(n)} \varphi )_V=(w^{(n)},\varphi )_{V_\pm(\cdot + a^{(n)}-a_\ast)},
\end{equation*}
provided that $n \in \mathbb{N}_\ast \setminus \mathbb{N} _{\sharp}$ and $n \in \mathbb{N} _\pm \cap\mathbb{N} _{\sharp},$ respectively. Since,
\begin{equation*}
\left| \gamma ^{-\frac{N+2s}{2}j_k ^{(n)}} f\left(\gamma ^{-j_k ^{(n)}}x + y_k ^{(n)} ,\gamma ^{\frac{N-2s}{2} j_k ^{(n)}} t  \right)\varphi \right| \leq C |t|^{2_s ^\ast -1},\quad\forall \,  k,n\text{ and }t,
\end{equation*}
thanks to Lebesgue Theorem, up to a subsequence, taking the limit in $k$ in the following identity
\begin{equation*}
I_\ast' (u_k) \cdot (d_k ^{(n)} \varphi ) = (v_k ^{(n)} , \varphi ) _V
- \int _{\mathbb{R}^N} \gamma ^{-\frac{N+2s}{2}j_k ^{(n)}} f\left(\gamma ^{-j_k ^{(n)}}x + y_k ^{(n)} ,\gamma ^{\frac{N-2s}{2} j_k ^{(n)}} v_k ^{(n)}  \right)\varphi \dx ,
\end{equation*}
we can conclude that $I'_\ast (w^{(1)})=(I_\pm ^{(n)} )'(w^{(n)}) = I_0' (w^{(n)})=0,$ in the corresponding index.

(iv) To conclude the proof, we prove now that $c(I_\ast ) < I^{(n_0)}_\pm (w^{(n_0)})$ or $c(I_\ast ) < I^{(n_0)}_\pm (w^{(n_0)}),$ where $n_0$ belongs to $\mathbb{N}_\ast \setminus \mathbb{N} _{\sharp}$ or $\mathbb{N} _\pm \cap\mathbb{N} _{\sharp}$ respectively. Define the path
$
\zeta^{(n_0)}(t) = t w^{(n_0)},\quad t\geq 0$ if $ n_0 \in \mathbb{N}_\ast \setminus \mathbb{N} _{\sharp}$ and $
\zeta^{(n_0)}(t) = t w^{(n_0)}(\cdot +a_\ast- a^{(n)}),\quad t\geq 0$ if $ n_0 \in \mathbb{N} _\pm \cap\mathbb{N} _{\sharp}.
$
Using \eqref{suficient_ast}--\eqref{suficient_ast_0} and Remark \ref{r_minimax}--(iv) we have that $\zeta^{(n_0)}$ belongs to $\Gamma_I$ with
\begin{equation*}
\begin{aligned}
&c(I_\ast) \leq \max _{t \geq 0}I_\ast (\zeta^{(n_0)}(t)) <I_0(\zeta^{(n_0)}(\bar{t}))\leq \max _{t \geq 0} I_0( \zeta^{(n_0)} (t) ) = I_0 (w^{(n_0)}),\quad\text{if } n_0 \in \mathbb{N}_\ast \setminus \mathbb{N} _{\sharp}.\\
&c(I_\ast) \leq \max _{t \geq 0}I_\ast (\zeta^{(n_0)}(t)) < I_\pm ^{(n)}( \zeta^{(n_0)} (\bar{t}) ) \leq \max _{t \geq 0} I_\pm ^{(n)}( \zeta^{(n_0)} (t) ) = I_\pm ^{(n)} (w^{(n_0)}),\quad\text{if } n_0 \in \mathbb{N} _\pm \cap\mathbb{N} _{\sharp},
\end{aligned}
\end{equation*}
 which together with \eqref{crit-right-hand} leads to a  contradiction ($\bar{t}$ is the maximum of $I_\ast (\zeta^{(n_0)}(t))$).

(v) We now assume only \eqref{suficient_ast}. Arguing in as before, we can prove 
$
u_k \rightarrow w^{(1)}\text{ in a subsequence}$ or $ c(I_\ast) = \max_{t \geq } I_\ast (\zeta^{(n_0)}(t)).
$
If the minimax level $c(I_\ast)$ is attained,  we apply \cite[Theorem~2.3]{lins} to obtain the existence of a critical point $u\in \zeta^{(n_0)} ([0, \infty))$ such that $I_\ast(u) = c(I_\ast).$
\end{proof}

{\it Acknowledgments.} {Research supported by INCTmat/MCT/Brazil, CNPq and CAPES/Brazil}


\end{document}